\documentclass[11pt,letterpaper]{article}
\usepackage{amsfonts, amsmath, amssymb, amscd, amsthm, color, graphicx, mathrsfs, wasysym, setspace, mdwlist, calc,float}
\usepackage{setspace}
\usepackage{hyperref}
\usepackage{hyperref}
\usepackage{tocloft}
\usepackage{xcolor}

\usepackage{hyperref}
\hypersetup{linktocpage}

\hypersetup{colorlinks,
    linkcolor={red!50!black},
    citecolor={blue!80!black},
    urlcolor={blue!80!black}}
\usepackage{float}

\newcommand{\e}{\varepsilon}
\newcommand{\NN}{\mathbb{N}}
\newcommand{\ZZ}{\mathbb{Z}}
\renewcommand{\S}{\mathcal S}

\renewcommand{\wr}{\operatorname{wr}}
\renewcommand{\d}{{\rm d}}
\renewcommand{\ll }{\langle\hspace{-.7mm}\langle }
\newcommand{\rr }{\rangle\hspace{-.7mm}\rangle }
\newcommand{\lab}{{\bf Lab}}

\renewcommand{\d}{{\rm d}}
\newcommand{\G}{\mathcal G}
\renewcommand{\phi}{\varphi}

\DeclareMathOperator{\Area}{Area}
\DeclareMathOperator{\Rad}{rad}
\DeclareMathOperator{\Stab}{Stab}
\DeclareMathOperator{\Cay}{Cay}

\newcommand\blfootnote[1]{%
  \begingroup
  \renewcommand\thefootnote{}\footnote{#1}%
  \addtocounter{footnote}{-1}%
  \endgroup
}

\newtheorem{thm}{Theorem}[section]
\newtheorem{cor}[thm]{Corollary}
\newtheorem{lem}[thm]{Lemma}
\newtheorem{prop}[thm]{Proposition}
\newtheorem{prob}[thm]{Problem}

\theoremstyle{definition}
\newtheorem{defn}[thm]{Definition}

\theoremstyle{remark}
\newtheorem{rem}[thm]{Remark}
\newtheorem{ex}[thm]{Example}

\newcommand{\la}{\langle}
\newcommand{\ra}{\rangle}

\newcommand{\Lab}{{\bf Lab}}
\newfont{\eufm}{eufm10}
\renewcommand{\tilde}{\widetilde}
\begin{document}

\title{\vspace*{-10mm} Isoperimetric inequalities in finitely generated groups}

\author{D. Osin \, E. Rybak}

\date{}
\maketitle
\vspace{-10mm}
\begin{abstract}
To each finitely generated group $G$, we associate a quasi-isometric invariant called the \emph{Dehn spectrum} of $G$. If $G$ is finitely presented, our invariant is closely related to the Dehn function of $G$, yet provides more
information by encoding the isoperimetric behaviour of $G$ at various scales. The main goal of this paper is to initiate the study of the Dehn spectrum of finitely generated (but not necessarily finitely presented) groups.  In particular, we compute the Dehn spectrum of small cancellation groups, certain wreath products, and free Burnside groups of sufficiently large odd exponent. We also address several natural questions concerning the structure of the poset of Dehn spectra. As an application, we show that there exist $2^{\aleph_0}$ pairwise non-quasi-isometric finitely generated groups of finite exponent.
\end{abstract}

\tableofcontents

\section{Introduction}
We begin by recalling the definition of the Dehn function of a finitely presented group. Let
\begin{equation}\label{Eq:Gpr}
G=\langle X\mid \mathcal R\rangle
\end{equation}
be a finite presentation of a group $G$. Thus, $G=F(X)/\ll \mathcal R\rr$, where $F(X)$ is the free group with the basis $X$ and $\ll \mathcal R\rr $ denotes the minimal normal subgroup containing the subset $\mathcal R \subseteq F(X)$. Every element $w\in \ll \mathcal R\rr$ admits a decomposition of the form
\begin{equation}\label{wprod}
w=\prod\limits_{i=1}^\ell f_i^{-1}R_i^{\e_i} f_i,
\end{equation}
where $f_i \in F(X)$, $R_i\in \mathcal R$, and $\e_i=\pm 1$ for all $i$. The minimal number $\ell$ over all such decompositions is denoted by $\Area_{\mathcal R}(w)$. The \emph{Dehn function of the presentation (\ref{Eq:Gpr})} is defined by the formula
$$
\delta_{G}(n)=\max\{ \Area_{\mathcal R}(w) \mid w\in \ll \mathcal R\rr, \; \| w\|\le n\},
$$
where $\| w\| $ denotes the length of the word $w$. In particular, every $w\in \ll \mathcal R\rr$ satisfies the \emph{isoperimetric inequality} $\Area_{\mathcal R}(w)\le \delta_{G}(\| w\|)$.

In the 1980s, Gromov \cite{Gro} characterized hyperbolic groups in terms of linear isoperimetric inequality. Ever since, the study of the Dehn function of finitely presented groups has been one of the central themes in geometric group theory. Research in this direction is strongly motivated by connections to geometric and algorithmic properties of groups. In particular, for every finitely presented group $G$, the following hold.

\begin{enumerate}
\item[(a)] Up to a natural equivalence relation, $\delta_{G}$ is independent of the choice of a particular finite presentation of $G$. Moreover, the equivalence class of the Dehn function is a quasi-isometric invariant of $G$.

\item[(b)] The equivalence class of $\delta_{G}$ can be used to detect the decidability of the word problem. Namely, the word problem in $G$ is decidable if and only if $\delta_{G}$ is equivalent to a recursive function.
\end{enumerate}

The Dehn function can be defined for any group presentation with a finite set of generators. However, properties (a) and (b) only hold for finitely presented groups. Moreover, the decidability of the word problem is not invariant under quasi-isometries of finitely generated groups \cite{Ers}. Thus, any invariant of finitely generated groups will unavoidably lack one of the above-mentioned properties.

In \cite{GI}, Grigorchuk and Ivanov suggested several analogs of the Dehn function whose computability for a (possibly infinite) recursive group presentation is equivalent to the decidability of the word problem. However, the functions considered in \cite{GI} are not quasi-isometric invariants; moreover, they depend on the choice of a particular presentation.

In this paper, we take a different approach and introduce the concept of a \emph{Dehn spectrum} of a finitely generated group, which enjoys quasi-isometric invariance. The Dehn spectrum is a function of three variables that generalizes the notion of the ordinary Dehn function to infinitely presented groups and refines it in the finitely presented case. In fact, we will work in the more general category of graphs but restrict our discussion to finitely generated groups in the introduction.

\begin{defn}\label{Def:main}
Let $G$ be a group generated by a finite set $X$. In particular, we have $G=F(X)/N_X$, for some normal subgroup $N_X$ of $F(X)$. For every $k\in \NN$, let $\mathcal S_k$ denote the set of all words of length at most $k$ in the alphabet $X\cup X^{-1}$ representing $1$ in $G$; that is,
\begin{equation}\label{Sk}
\S_k=\{w\in N_X \mid \| w\|\le k\}.
\end{equation}
The \emph{Dehn spectrum of $G$ with respect to $X$} is the function defined on the set of all triples $(k,m,n)\in \NN^3$ satisfying the inequality $m\ge k$ by the formula
\begin{equation}\label{fkmn}
f_{G,X}(k,m,n)=\max\{ \Area_{\S_m}(w) \mid w\in \ll \S_{k} \rr,\; \|w\|\le n\}.
\end{equation}
Note that $\Area_{\S_m}(w)$ is well-defined since $m\ge k$ and $w\in \ll \S_{k}\rr \le \ll \S_m\rr$.
\end{defn}

We consider an important particular case.

\begin{rem}\label{Ex:fp}
Suppose that a group $G$ is finitely presented, that is, $N_X=\ll \mathcal{R} \rr$ for some finite subset $\mathcal{R}$ of $N_X$. Let $M=\max\{ \| R\| \mid R\in \mathcal R\}$. Then $\ll \S_M\rr =\ll \S_{M+1}\rr =\ldots $ and for every $k\ge M$, we have $f_{G,X}(k,m,n)=\delta _{\langle X\mid \mathcal S_m\rangle }(n)$, where $\delta _{\langle X\mid \mathcal S_m\rangle }$ is the Dehn function of $G$ computed using the finite presentation $\langle X\mid \mathcal S_m\rangle $. Thus, the Dehn spectrum provides more information about $G$ than any single Dehn function; in fact, it encodes the whole spectrum of the isoperimetric behaviour of $G$ at various scales, hence the name.
\end{rem}

We denote by $\mathcal F$ the set of all functions $\{ (k,m,n)\in \NN^3\mid m\ge k\} \to \mathbb R$ that are non-decreasing in the first and third arguments, and non-increasing in the second argument. Clearly, $f_{G,X}\in \mathcal F$ for any group $G$ generated by a finite set $X$.

\begin{defn}\label{Def:ord}
Given two functions $f, g\colon \NN^3\to \mathbb R$, we write $f\preccurlyeq g$ if there exists $C\in \mathbb N$ such that
$$
f(k, Cm, n)\le Cg(Ck, m, Cn) + Cn/m + C
$$
for all $(k, m, n)\in \NN^3$ such that $m\ge Ck$. Further, we write $f\sim g$ and say that $f$ and $g$ are \emph{equivalent} if $f\preccurlyeq g$ and $g\preccurlyeq f$.
\end{defn}

It is straightforward to verify that $\sim$ is indeed an equivalence relation on the set $\mathcal F$. Our first result establishes the quasi-isometric invariance of the Dehn spectrum.

\begin{thm}\label{Thm:QI}
Let $G$ and $H$ be groups generated by finite sets $X$ and $Y$, respectively. If $G$ and $H$ are quasi-isometric, then $f_{G,X} \sim f_{H,Y}$. In particular, the equivalence class of the Dehn spectrum of a finitely generated group is independent of the choice of a particular generating set.
\end{thm}

\begin{rem}
A thoughtful reader may wonder why we defined the Dehn spectrum of a group as a function of three variables. Indeed, it might be tempting to consider the function $g_{G,X}(k,n)=f_{G,X}(k,k,n)$. The reason behind our choice of the more complicated definition is that the function $g_{G,X}$ may not satisfy any reasonable analog of Theorem~\ref{Thm:QI}.
\end{rem}

In this paper, we often consider Dehn spectra up to the equivalence relation introduced above. In these cases, we omit the generating set from the notation and simply denote the Dehn spectrum of a finitely generated group $G$ by $f_G$.

As we have already seen in Remark \ref{Ex:fp}, if $G$ is a finitely presented group, the Dehn spectrum of $G$ essentially reduces to a function of two variables closely related to the Dehn function of $G$. More precisely, we say that a function $f\in \mathcal F$ is \emph{essentially independent} of $k$ if $f(k, m,n)$ is equivalent to a function that depends on $m$ and $n$ only. We prove the following.

\begin{thm}\label{Thm:FP}
For every finitely presented group $G$, the function $f_G(k,m,n)$ is essentially independent of $k$ and we have
\begin{equation}\label{fpspectrum}
\frac{\delta_G(n)}{\delta_G(m)}\preccurlyeq f_G(k,m,n)\preccurlyeq\frac{\delta_G(n)}m,
\end{equation}
where $\delta_G(n)$ is the Dehn function of some finite presentation of $G$.
\end{thm}

\begin{rem}\label{Rem:pres}
Here we think of the lower and upper bounds in (\ref{fpspectrum}) as functions of triples $(k,m,n)\in \NN^3$ independent of $k$. We need to be careful with the choice of a finite presentation of $G$ to make sure that $\delta_G(n)/\delta_G(m)$ is well-defined. However, both inequalities in (\ref{fpspectrum}) hold for any finite presentation of $G$ such that the corresponding Dehn function satisfies $\delta_G(m)>0$ for all $m\in \NN$  (see Corollary \ref{Cor:fp}.)
\end{rem}

Despite being closely related to the Dehn function, the Dehn spectrum of a finitely presented group $G$ is a finer invariant and is not determined by $\delta_G$ up to the $\sim$ relation as is shown in Proposition \ref{Prop:AC}. 

We discuss several particular examples illustrating various behaviours of the Dehn spectrum of finitely presented groups.
If $G$ is hyperbolic, the lower and upper bounds in (\ref{fpspectrum}) coincide. Conversely, the equivalence $f_G(k,m,n)\sim n/m$ for a finitely presented group $G$ implies hyperbolicity of $G$ (see Remark \ref{Ex:fp}).

\begin{defn}
We say that a finitely generated group $G$ has \emph{linear Dehn spectrum} if $f_G(k,m,n)\sim n/m$.
\end{defn}

\begin{cor}\label{Cor:Hyp}
A finitely presented group is hyperbolic if and only if its Dehn spectrum is linear.
\end{cor}

The right inequality in (\ref{fpspectrum}) is strict for any non-hyperbolic finitely presented group (see Corollary \ref{Cor:hyp}). In contrast, the left inequality often becomes an equivalence. For example, combining Theorem \ref{Thm:FP} with results of Papasoglu \cite{Pap}, we obtain the following.

\begin{cor}\label{Cor:Quad}
For any finitely presented group $G$ with quadratic Dehn function, we have $f_G(k,m,n)\sim (n/m)^2$.
\end{cor}

In particular, we have $f_G(k,m,n)\sim (n/m)^2$ for any non-hyperbolic semihyperbolic group in the sense of Alonso and Bridson \cite{AB}.

In fact, it is not easy to construct a finitely presented group such that the left inequality in (\ref{fpspectrum}) is strict. The simplest construction of such examples is based on the following result, which is interesting on its own right.

\begin{prop}\label{Prop:AC}
For any finitely presented group $G$, the following conditions are equivalent.
\begin{enumerate}
\item[(a)] All asymptotic cones of $G$ are simply connected.
\item[(b)] There exists a function $g\colon \mathbb Q_+\to \mathbb R$ such that $f_G(k,m,n) \preccurlyeq g(n/m)$.
\item[(c)] There exists $\alpha\in (0, \infty)$ such that $f_G(k,m,n) \preccurlyeq (n/m)^\alpha $
\end{enumerate}
\end{prop}

In fact, this proposition follows from Gromov's characterization of groups with simply connected asymptotic cones in terms of the so-called loop division property (see \cite[Sec. 5.F]{Gro93} and \cite{Dru, Pap} for further details). As a corollary, we have  $f_G(k,m,n)\not\sim\delta_G(n)/\delta_G(m)$ for any finitely presented group $G$ with a polynomial Dehn function such that not all asymptotic cones of $G$ are simply connected. The first groups of this sort were constructed by Bridson \cite{Bri}, and a somewhat easier example was later found by Olshanskii and Sapir \cite{OS}. We use their example and Proposition \ref{Prop:AC} to provide examples of two finitely presented groups with cubic Dehn functions and inequivalent Dehn spectra (see Corollary \ref{Cor:H3}).

As we pass from finitely presented to finitely generated groups, many natural questions arise. In this paper, we discuss three of them, highlighting the striking difference between the finitely presented and finitely generated cases.
\begin{enumerate}
\item[(Q1)] Is every finitely generated group with linear Dehn spectrum hyperbolic?
\item[(Q2)] What is the cardinality of the set of equivalence classes of the Dehn spectrum of finitely generated groups?
\item[(Q3)] Does there exist a finitely generated group $G$ whose Dehn spectrum essentially depends on the first argument?
\end{enumerate}

We begin with the first question. The theorem below provides many examples of non-hyperbolic groups with linear Dehn spectrum.

\begin{thm}\label{Thm:Lin}
The following groups have a linear Dehn spectrum.
\begin{enumerate}
\item[(a)] Wreath products of the form $K{\,\rm wr\,} \mathbb Z$, where $K$ is an arbitrary finite group.
\item[(b)] Finitely generated (but not necessarily finitely presented) $C^\prime(1/6)$-groups.
\item[(c)] Finitely generated free Burnside groups of sufficiently large odd exponent.
\end{enumerate}
\end{thm}

Further, we show that ``generic" limits of hyperbolic groups provide the answers to (Q2) and (Q3).  More precisely, we consider the Polish space $\G_n$ of $n$-generated marked groups. Informally, $\G_n$ is the set of pairs $(G,X)$, where $G$ is a group and $X$ is an ordered generating set of $G$ of cardinality $n$, with the topology induced by the local convergence of Cayley graphs (see Section \ref{Sec:aaifgg} for details). The relation $\preccurlyeq$ introduced in Definition \ref{Def:ord} induces the following quasi-order on the spaces $\G_n$:
$$(G,X)\preccurlyeq _{\mathrm{Ds}} (H,Y)\;\;\; \Longleftrightarrow \;\;\; f_{G,X}\preccurlyeq f_{H,Y},$$
where $\mathrm{Ds}$ stands for the Dehn spectrum.

Recall that a hyperbolic group is said to be \emph{elementary} if it contains a cyclic subgroup of finite index. Let $\mathcal H_n$ denote the subset of $\G_n$ consisting of non-elementary hyperbolic marked groups, and let  $\overline{\mathcal H}_n$ be the closure of $\mathcal H_n$ in $\G_n$. As usual, we say that a certain property \emph{holds for a generic element} of a Polish space $P$ if it holds for every element of a comeager subset of $P$. Note that $\overline{\mathcal H}_n$ is a Polish space being a closed subset of the Polish space $\mathcal G_n$, and thus comeager subsets of $\overline{\mathcal H}_n$ are non-empty by the Baire category theorem.

\begin{thm}\label{Thm:HL}
For every integer $n\ge 2$, the following hold.
\begin{enumerate}
\item[(a)] The space $\overline{\mathcal H}_n$ contains $2^{\aleph_0}$ pairwise $\preccurlyeq_{\mathrm{Ds}}$--incomparable elements.
\item[(b)] For a generic element $(G, X)$ of $\overline{\mathcal{H}}_n$, the Dehn spectrum $f_{G,X}(k,m,n)$ is not equivalent to any function that depends on $m$ and $n$ only.
\end{enumerate}
\end{thm}

In particular, the set of equivalence classes of the Dehn spectrum of finitely generated groups has the cardinality of the continuum. Our proof of Theorem \ref{Thm:HL} is non-constructive and employs ideas from descriptive set theory. It is inspired by the proof of the fact that $\overline{\mathcal H}_n$ has $2^{\aleph_0}$ quasi-isometry classes presented in \cite{MOW}. Another ingredient used in the proof is group-theoretic Dehn filling.

Finally, we illustrate the usefulness of the Dehn spectrum by considering an application to groups of finite exponent. Recall that the famous Burnside problem asks whether every finitely generated group of a finite exponent $N$ is finite. The negative answer was first obtained by Novikov and Adian for all odd $N\ge 4381$ in the series of papers \cite{NA1, NA2, NA3}; an improved version of the original proof for odd $N\ge 665$ can be found in \cite{Adi}. In \cite{Ols83}, Olshanskii suggested a much simpler geometric proof for odd $N>10^{10}$. For all sufficiently large even exponents, the negative solution was obtained by Ivanov \cite{Iva} and Lysenok \cite{Lys}. More recently, Coulon \cite{Cou} provided yet another geometric proof following the ideas of Delzant and Gromov \cite{DG}; his proof explicitly uses hyperbolic geometry and works for all sufficiently large exponents, odd and even.

Despite this progress, very little is known about the large-scale geometry of finitely generated groups of finite exponent. Our last result shows that such groups can be very diverse from the geometric point of view. We state our theorem in a simplified form here and refer to Theorem \ref{Thm:main-full} for more details.

\begin{thm}\label{Thm:FE}
There exist $2^{\aleph_0}$ groups of finite exponent with pairwise incomparable Dehn spectra. In particular, there are $2^{\aleph_0}$ quasi-isometry classes of such groups.
\end{thm}

The existence of $2^{\aleph_0}$ pairwise non-isomorphic finitely generated groups of finite exponent was established by Olshanskii (see \cite[Theorem 28.7]{book}). However, the existence of at least 2 distinct quasi-isometry classes of infinite finitely generated groups of finite exponent seems to be new. The proof of Theorem \ref{Thm:FE} combines descriptive ideas similar to those used in the proof of Theorem~\ref{Thm:HL} with Olshanskii's small cancellation technique developed in \cite{Ols83,book}.

The paper is organized as follows. In Section \ref{Sec:EDBP}, we define the Dehn spectrum of a graph and prove a generalization of Theorem \ref{Thm:QI} in this context; Theorem \ref{Thm:FP} and its corollaries are also proved there. The first two parts of Theorem \ref{Thm:Lin} are obtained in Section~\ref{Sec:GwLS}.  Section \ref{Sec:DA} is devoted to the proof of Theorem \ref{Thm:HL}. The proofs of part (c) of Theorem \ref{Thm:Lin} and Theorem \ref{Thm:FE} are given in Section \ref{Sec:GFE}.

\paragraph{Acknowledgments.} The authors are grateful to O. Kulikova for the constant attention to this work. We would also like to thank the anonymous referees for careful reading of our manuscript and useful suggestions. The first author has been supported by the NSF grants DMS-1612473 and DMS-1853989.


\section{Basic properties of Dehn spectrum}\label{Sec:EDBP}


\subsection{Notation.} We begin by establishing the notation and terminology used throughout the paper. 

Let $G$ be a group generated by a set $X$. Given a word $w$ in the alphabet $X \cup X^{-1}$, we denote by $\|w\|$ (respectively, $|w|_X$) the number of letters in $w$ (respectively, the word length of the element of $G$ represented by $w$ with respect to the generating set $X$). For two (possibly non-reduced) words $u$ and $v$ in the alphabet $X\cup X^{-1}$, we write $u\equiv v$ if they are equal as words and $u=v$ if they represent the same element of the group.

Let $K$ be a $CW$-complex. By $K^{(n)}$ we denote the $n$-skeleton of $K$. If $K$ is $2$-dimensional, we use the terms \emph{vertices, edges,} and \emph{faces} for $0$-cells, $1$-cells, and $2$-cells of $K$, respectively. By a \emph{path} $p$ in a $CW$-complex $K$ we always mean a combinatorial path in $K^{(1)}$; we denote by $\ell(p)$, $p_-$, and $p_+$ its length, origin, and terminal vertex, respectively; $p$ is a \emph{loop} if $p_-$=$p_+$.

A \emph{graph} is a $1$-dimensional $CW$-complex. Thus, we allow graphs to have multiple edges as well as edges that join a vertex to itself. We use the notation $V(\Gamma)$ and $E(\Gamma)$ for the sets of vertices and edges of $\Gamma$, respectively. Throughout this paper, we think of graphs as metric spaces with respect to the length metric induced by identifying interiors of edges with the interval $(0,1)$.

Suppose that $\Gamma$ is a finite connected plane graph. Then the interior of every face $\Pi$ of $\Gamma$ is homeomorphic to a disc and the boundary of $\Pi$, denoted by $\partial \Pi$, is a loop in $\Gamma$. The boundary of the exterior face of $\Gamma$ is denoted by $\partial \Gamma$. Henceforth, the term \emph{face} will mean an interior face in this context. To every such graph, we can associate a planar $CW$-complex obtained by gluing a $2$-cell to the boundary of every (interior) face of $\Gamma$. This complex is connected and simply connected, but may not be homeomorphic to a disk.

We assume the reader to be familiar with the notions of a Cayley graph and a van Kampen diagram and refer to \cite[Chapter III]{LS} and \cite[Section 11]{book} for details. The Cayley graph of a group $G$ with respect to a generating set $X$ is denoted by $\Cay(G,X)$. If $p$ is a path in an oriented labelled graph, we denote by $\Lab(p)$ its label. Labels of boundaries of van Kampen diagrams and their faces are usually considered up to a cyclic shift and taking inverses. E.g., we write $\Lab (\partial \Delta)\equiv w$ for a van Kampen diagram $\Delta$ to mean that $w$ reads along $\partial \Delta $ starting from a certain point in an appropriate direction.


\subsection{Dehn spectrum of a graph}


\begin{defn} Let $\Gamma$ and $\Delta $ be any graphs. We say that $\phi\colon \Delta \to \Gamma$ is a \emph{generalized combinatorial map} if $\phi $ is continuous, maps vertices of $\Delta$ to vertices of $\Gamma$, and the image of every edge of $\Delta$ under $\phi $ is either a vertex or an edge of $\Gamma$.
\end{defn}

Note that every generalized combinatorial map sends paths to paths and loops to loops.

\begin{defn}\label{Def:k-fil}
Let $c$ be a loop in a graph $\Gamma$ and let $k\in \mathbb N$. A \emph{combinatorial $k$-filling} of $c$ is a generalized combinatorial map $\phi \colon D\to \Gamma$, where $D$ is a finite connected plane graph, such that the following conditions hold.
\begin{enumerate}
\item[(a)] There exists a vertex $o\in \partial D$ such that $\phi (o)=c_-=c_+$ and the loop $\phi(\partial D)$ starting at $\phi(o)$ coincides with $c$.
\item[(b)] For every (interior) face $F$ of $D$, we have $\ell(\phi(\partial F))\le k$.
\end{enumerate}
The \emph{area} of $\phi$, denoted by $\Area(\phi)$, is the number of faces of $D$. A loop $c$ in $\Gamma $ is said to be \emph{$k$-contractible} if it admits a combinatorial $k$-filling. The set of all $k$-contractible loops in $\Gamma $ will be denoted by $C_k(\Gamma)$. For $c\in C_k(\Gamma)$, we define $\Area_k(c)$ to be the minimum of areas of all combinatorial $k$-fillings of $c$.
\end{defn}

Two remarks are in order. First, we do not insist that $\partial D$ is a simple path. Second, it is important that $\phi(\partial D)$ and $c$ coincide as paths (not as subsets of $\Gamma$) in (a). E.g., if $c_0$ is any loop in $\Gamma$ and $c=c_0c_0$, then $\phi(\partial D)$ must also wind twice around $c_0$.

\begin{defn}\label{Def:IPG}
Let $\mathbb T=\{ (k,m,n)\in \mathbb N\times \mathbb N\times\mathbb N\mid m\ge k \}$. The \emph{Dehn spectrum} $f_\Gamma\colon \mathbb T \to \mathbb N\cup\{\infty\}$ of a connected graph $\Gamma $ is defined by the formula
$$
f_\Gamma (k,m,n)=\sup\{ \Area_{m}(c)\mid c\in C_k(\Gamma),\; \ell(c)\le n\}.
$$
Note that $\Area_{m}(c)$ is well-defined for $m\ge k$ since $C_k(\Gamma)\subseteq C_m(\Gamma)$. We say that $\Gamma $ has \emph{finitary} Dehn spectrum if $f_\Gamma $ takes only finite values.
\end{defn}

\begin{rem}
Instead of the combinatorial approach taken here, we could work with a more general notion of filling and area (see \cite[Section III.H.2]{BH} or \cite[Section 9.7]{DK}), which would allow us to define Dehn spectra of arbitrary metric spaces.
\end{rem}

\begin{lem}\label{Lem:finitary}
Every vertex-transitive graph of finite degree has a finitary Dehn spectrum. In particular, this is true for any Cayley graph of a group with respect to a finite generating set.
\end{lem}

\begin{proof}
The claim follows immediately from the observation that for every $n\in \NN$, there are only finitely many $Aut(\Gamma)$-orbits of loops of length $n$ in $\Gamma$.
\end{proof}

Recall that $\mathcal F$ denotes the set of all functions $f\colon \mathbb T\to \mathbb R$ such that $f$ is non-decreasing in the first and third arguments, and non-increasing in the second argument. It is not very difficult to see that $f_{\Gamma }\in \mathcal F$ for any graph $\Gamma$ whenever $f_\Gamma$ is finitary. As in the case of groups, we consider finitary Dehn spectra of graphs up to the equivalence relation introduced in Definition \ref{Def:ord}. For completeness, we record the following.

\begin{prop}\label{Prop:F}
\begin{enumerate}
\item[(a)] Let $\Gamma $ be a graph. If $f_\Gamma$ is finitary, then $f_\Gamma\in \mathcal F$.
\item[(b)] The relation $\preccurlyeq$ is a preorder and $\sim$ is an equivalence relation on $\mathcal F$.
\end{enumerate}
\end{prop}

\begin{proof}
Part (a) follows immediately from the definition. To prove transitivity of $\preccurlyeq$, suppose that $f\preccurlyeq g\preccurlyeq h$ for some $f,g,h\in \mathcal F$. By definition, there exist $C_1, C_2\in \NN$ such that
$$
f(k, C_1m, n)\le C_1g(C_1k, m, C_1n)+C_1 n/m +C_1
$$
and
$$
g(k, C_2m, n)\le C_2h(C_2k, m, C_2n)+C_2 n/m +C_2
$$
for all $(k,m,n)\in \NN^3$ satisfying $m\ge C_1k$ and $m\ge C_2k$, respectively. Therefore, we have
$$
f(k, C_1C_2m, n)\le  C_1\big( C_2h(C_1C_2k, m, C_1C_2n)+C_2 C_1n/m + C_2\big) + C_1n/(C_2m) +C_1
$$
for all triples $(k,m,n)\in \NN^3$ satisfying $m\ge C_1C_2k$. Since the functions $f$ and $h$ are non-increasing (respectively, nondecreasing) in the second (respectively, first and third) argument, we have $f(k,Cm,n)\le Ch(Ck,m,Cn)+Cn/m +C$, where $C=C_1(C_1C_2+1)$. Thus, $\preccurlyeq$ is a preorder and, therefore, $\sim $ is an equivalence relation.
\end{proof}

Our next goal is to relate Definition \ref{Def:IPG} to the one given in the introduction. To this end, we will use the van Kampen lemma, which states that a word $w\in F(X)$ represents $1$ in the group given by a presentation $\langle X\mid \mathcal R\rangle$ if and only if there exists a van Kampen diagram $\Delta $ over this presentation such that $\Lab(\partial \Delta)\equiv w$ (see \cite{LS}). It follows immediately from the proof that the minimal number of faces in such a diagram equals $\Area_{\mathcal R} (w)$. Furthermore, for every loop $c$ in the Cayley graph $\Cay(G,X)$ labelled by $w$, there is a canonical graph homomorphism  $\Delta^{(1)}\to \Cay(G,X)$ preserving labels and orientation that sends $\partial \Delta$ to $c$.

\begin{prop}\label{Prop:f=f}
For every group $G$ generated by a finite set $X$, we have
\begin{equation}\label{Eq:ips=}
f_{\Cay(G,X)}(k,m,n)= f_{G,X}(k,m,n).
\end{equation}
\end{prop}

\begin{proof}
Let $c$ be a loop in $\Cay(G,X)$ and let $w\equiv \Lab(c)$. Assume first that $c\in C_k(\Gamma)$ for some $k\in \NN$. Let $\phi\colon D\to \Cay(G,X)$ be a minimal area combinatorial $m$-filling of $c$ for some $m\ge k$. We orient and label every $e\in E(D)$ as follows. If $\phi(e)$ is a vertex, we orient $e$ arbitrarily and label it with the empty word; if $\phi(e)=f\in E(\Cay(G,X))$ we endow $e$ with the same orientation and label as $f$. Further, we fill in all faces of the plane graph $D$ by gluing $2$-cells along their boundaries. By part (b) of Definition \ref{Def:k-fil}, we have $\Lab(\Pi)\in \mathcal S_m$ for every face $\Pi$ of the obtained $2$-complex $\Delta$ (see (\ref{Sk})).

Thus, $\Delta$ is almost a van Kampen diagram over $\langle X\mid \mathcal S_m\rangle$, the only difference is that we may have edges labelled by the empty word. Cutting $\Delta $ into the union of individual faces just as in the proof of the ``if" part of the van Kampen lemma, we obtain a decomposition (\ref{wprod}), where $\ell =  \Area_m(c)$ and $R_i\in \mathcal S_m$ for all $i$. In particular, $w\in \ll \mathcal S_k\rr$ and, for any $m\ge k$, we have $\Area_{\mathcal S_m} (w)\le \Area_m(c)$.

Now, suppose that $w\in \ll \mathcal S_k\rr$. For every $m\ge k$, there exists a van Kampen diagram $\Delta$ over $\langle X \mid \mathcal S_m\rangle$ with at most $\Area_{\mathcal S_m}(w)$ faces and the boundary label $\Lab(\Delta)\equiv w$. The canonical map $\Delta^{(1)} \to \Cay(G,X)$ sending $\partial \Delta$ to $c$ is a combinatorial $m$-filling of $c$ of area at most $\Area_{\mathcal R}(w)$. In particular, $c$ is $k$-contractible and $\Area_{m}(c)\le \Area_{\mathcal S_m}(w)$. Combining this with the opposite inequality proved above, we obtain (\ref{Eq:ips=}).
\end{proof}


\subsection{Quasi-isometric invariance}\label{Sec:QI}


The main goal of this section is to show that the Dehn spectrum is a quasi-isometric invariant of connected graphs. Recall that a map $\alpha \colon S\to T$ between metric spaces $(S, \d_S)$, $(T, \d_T)$ is \emph{$(K,L)$-coarsely Lipschitz} for some constants $K, L\ge 0$ if \begin{equation}\label{C-Lip}
\d_T(\alpha(x), \alpha(y)) \le K\d_S (x,y)+ L
\end{equation}
for all $x,y\in S$. Metric spaces $(S, \d_S)$ and $(T, \d_T)$ are \emph{quasi-isometric} if there exists $L\ge 0$ and $(L,L)$-coarsely Lipschitz maps $\alpha\colon S\to T$ and $\beta\colon T\to S$ such that $\alpha$ and $\beta$ are \emph{$L$-coarse inverses} of each other; that is,
\begin{equation}\label{ci}
\sup\limits_{t\in T} \d_T (\alpha\circ \beta(t), t)\le L \;\;\; {\rm and }\;\;\; \sup\limits_{s\in S} \d_S (\beta\circ \alpha(s), s)\le L.
\end{equation}

We begin by describing a general construction that will be used several times in our paper.

\begin{defn}\label{Def:im}
Let $\Gamma$, $\Delta $ be two graphs, $\alpha\colon V(\Gamma )\to V(\Delta )$ a map between the sets of their vertices, $c=e_1\ldots e_r$ a path in $\Gamma$, where $e_1, \ldots, e_r\in E(\Gamma)$. We say that a path $d$ in $\Delta $ is a \emph{coarse $\alpha$-image} of $c$ if it decomposes as $d=d_1\ldots d_r$, where $d_i$ is a geodesic path in $\Delta $ and $\alpha ((e_i)_\pm)=(d_i)_\pm$ for all $i$; if $c$ is not simple, we additionally require that if $e_i=e_j^{\pm 1}$ for some $i\ne j$, then $d_i=d_j^{\pm 1}$.
\end{defn}

\begin{lem}\label{Lem:indf}
Let $\Gamma$, $\Delta $ be two graphs, $\alpha\colon \Gamma\to \Delta$ a $(K,L)$-coarsely Lipschitz map, where $K,L\in \NN$, sending vertices of $\Gamma$ to vertices of $\Delta$. For any path $c$ in $\Gamma$ and any coarse $\alpha$-image $d$ of $c$ in $\Delta$, the following hold.
\begin{enumerate}
\item[(a)] $\ell(d)\le (K+L)\ell(c)$.
\item[(b)] If $c\in C_k(\Gamma)$, then $d\in C_{(K+L)k}(\Delta)$ and  $\Area_{(K+L)k}(d)\le \Area_k(c)$.
\end{enumerate}
\end{lem}

\begin{proof}
In the notation of Definition \ref{Def:im}, let $e_i$ be an edge of $c$. The corresponding subpath $d_i$ of $d$ has length
$$
\ell(d_i)=\d_{\Delta} (\alpha((e_i)_-), \alpha((e_i)_+)) \le K\d_\Gamma ((e_i)_-,(e_i)_+)+L \le K+L.
$$
This implies (a).

Further, every combinatorial filling $\phi \colon C \to \Gamma $ of $c$ can be transformed to a combinatorial filling $\psi \colon D \to \Delta $ of $d$ as follows. Let $e$ be an edge of $C$. If $\phi (e)$ is a vertex of $\Gamma$, we do not modify $e$ and define $\psi(e)=\phi(e)$. If $\phi (e)=e_i$ for some $i$ and $\ell(d_i)>0$, we subdivide $e$ into $\ell(d_i)$ edges and let $\psi$ map the resulting path to $d_i$ in the obvious way. For any other edge $e$ of $C$, we fix any geodesic path $p$ in $\Delta$ connecting $\alpha\circ \phi(e_-)$ to $\alpha\circ\phi(e_+)$, subdivide $e$ into $\ell(p)$ edges, and let $\psi$ map the obtained path to $p$. Let $D$ denote the resulting plane graph. Since every edge of $C$ gets subdivided into at most $K+L$ edges, $\psi\colon D \to \Delta $ is a combinatorial $(K+L)k$-filling of $d$, and it has the same area as $\phi$. Applying this construction to a combinatorial $k$-filling of $c$ of minimal area, we obtain (b).
\end{proof}

\begin{thm}\label{Thm:QIG}
Let $\Gamma $ and $\Delta$ be connected quasi-isometric graphs. Suppose that $\Gamma$ has a finitary Dehn spectrum. Then so does $\Delta$, and we have $f_\Gamma\sim f_\Delta$.
\end{thm}

\begin{proof}
Let $\alpha\colon \Gamma\to \Delta $ and $\beta\colon \Delta\to \Gamma$ be $L$-coarse inverse, $(L,L)$-coarsely Lipschitz maps for some $L\in \NN$. Without loss of generality, we can assume that $\alpha$ and $\beta$ send vertices to vertices. Let $k$ be a positive integer, $c$ a $k$-contractible loop in $\Gamma $, $d$ a coarse $\alpha$-image of $c$ in $\Delta$. By Lemma \ref{Lem:indf}, $d$ is $2Lk$-contractible.

Further, let $m\ge 2Lk$ and let $\tilde c$ be a coarse $\beta$-image of $d$ in $\Gamma$.  Applying Lemma \ref{Lem:indf} to a combinatorial $m$-filling of $d$, we obtain a combinatorial $2Lm$-filling $\tilde \phi\colon \tilde D\to \Gamma$ of $\tilde c$ of area
$$
\Area (\tilde\phi)=\Area_{m}(d) \le f_\Delta (2Lk, m, \ell(d)) \le f_\Delta (2Lk, m, 2L\ell(c)).
$$

\begin{figure}
  \begin{center}
 \hspace{18mm} 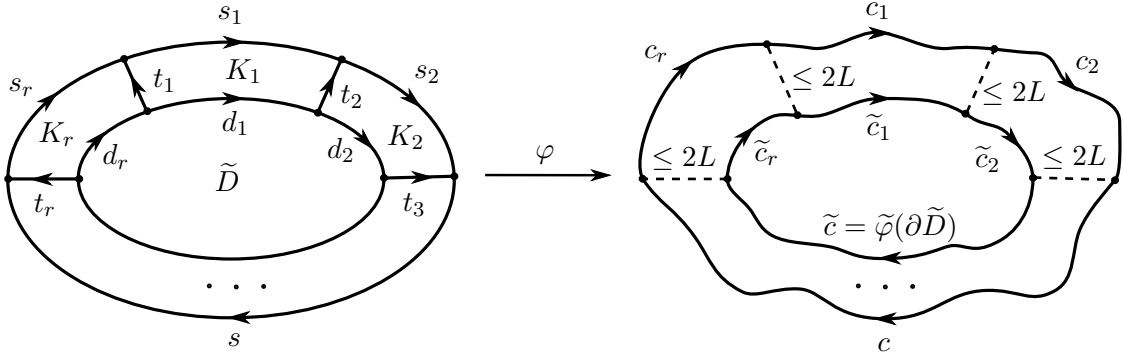
  \end{center}
  \vspace{-3mm}
  \caption{Constructing a combinatorial filling of $c$.}\label{figbord}
\end{figure}

In order to obtain a filling of $c$, we modify $\tilde D$ and $\tilde \phi$ as follows. Let $c=c_1\ldots c_r$, where
$$
r\le  \lceil\ell(c)/m\rceil$$
and
$$
1\le \ell(c_i)\le m
$$
for all $i=1, \ldots, r$. Let $\tilde c =\tilde c_1\ldots \tilde c_r$ and $\partial \tilde D=d_1\ldots d_r$, where
$$
(\tilde c_i)_\pm =\beta\circ \alpha((c_i)_\pm)\;\;\; {\rm and}\;\;\; \tilde\phi(d_i)=\tilde c_i, \;\;\; i=1, \ldots, r.
$$
Since $\alpha$ and $\beta $ are $L$-coarse inverse, we have $\d_\Gamma((c_i)_-, (\tilde c_i)_-)\le L$. Let $s=s_1\ldots s_r$ be a simple combinatorial loop (i.e., a graph homeomorphic to $\mathbb S^1$) embedded in the Euclidean plane, where $\ell(s_i)=\ell(c_i)$ for all $i$. We embed $\tilde D$ in the open disk bounded by $s$ so that $\partial \tilde D$ and $s$ are oriented in the same way. For  every $i$, we connect $(d_i)_-$ to $(s_i)_-$ by a combinatorial path $t_i$ of length $\d_\Gamma((c_i)_-, (\tilde c_i)_-)$ so that the interior of $t_i$ is disjoint from $s$, $\tilde D$, and $t_j$ for $j\ne i$ (see Fig. \ref{figbord}). Finally, we extend $\tilde\phi$ to a map $\phi \colon D\to \Gamma$ by sending each $s_i$ to $c_i$ and each $t_i$ to a geodesic path connecting $(\tilde c_i)_-$ to $(c_i)_-$ in $\Gamma$.

Let $K_i$ denote the face of $D$ bounded by $t_is_it_{i+1}^{-1}d_i^{-1}$ (indices are modulo $r$). Applying part (a) of Lemma \ref{Lem:indf} twice, we obtain $\ell(\tilde c_i)\le 4L^2 \ell(c_i)$ for every $i$. Hence,
$$
\begin{array}{rcl}
\ell(\phi(\partial K_i)) & = & \ell(\phi(s_i))+\ell(\phi (d_i))+\ell(\phi(t_i))+ \ell(\phi(t_{i+1}))\\&&\\
& = & \ell(c_i) + \ell(\tilde c_i) + 2L \\&&\\
& \le & \ell(c_i) + 4L^2 \ell(c_i) + 2L \\&&\\
& \le & (4L^2+2L+1)m.
\end{array}
$$
Therefore, $\phi$ is a combinatorial $Cm$-filling of $c$, where $C=4L^2+2L+1$. Note that
$$
\Area_{Cm}(c)\le \Area(\phi) \le \Area(\tilde\phi) +r\le f_\Delta (2Lk, m, 2L\ell(c)) + \ell(c)/m +1.
$$
This implies the inequality $f_\Gamma(k, m, n)\preccurlyeq f_\Delta (k,m,n)$. The proof of the opposite inequality is symmetric.
\end{proof}

Combining Lemma~\ref{Lem:finitary} and Theorem~\ref{Thm:QIG}, we could define, up to the equivalence relation introduced above, the Dehn spectrum of any metric space $M$ that is quasi-isometric to a locally finite, vertex-transitive graph $\Gamma$ by setting $f_M = f_\Gamma$. In particular, every compactly generated, totally disconnected, locally compact group admits a well-defined,  up to $\sim$, finitary Dehn spectrum, defined via its Cayley--Abels graph (see \cite{Led} for the definition and a survey of known results). Another natural direction is to consider Riemannian settings, which would allow one to define finitary Dehn spectra of Lie groups. Furthermore, it would be interesting to develop a unified approach in the broader framework of compactly generated locally compact groups. Note that the formal combination of Lemma~\ref{Lem:finitary} and Theorem~\ref{Thm:QIG} would not suffice in the latter two cases, as compactly generated locally compact groups (and even Lie groups) may not be quasi-isometric to locally finite vertex-transitive graphs; for an example, see the appendix to \cite{Cor}.

To keep the length of the paper under control, we work at the level of generality most suitable for our main goal -- the study of discrete groups -- and leave the development of the theory in more topological settings to interested readers.


\subsection{The case of finitely presented groups}\label{Sec:FP}


Our next goal is to prove Theorem \ref{Thm:FP}. To this end, we fix a finitely presented group $G$ and consider a special finite presentation of $G$.
\begin{defn}\label{Gpres}
Let $G=\langle X\mid \mathcal R\rangle$ be a finite presentation of a group $G$ satisfying the following two conditions:
\begin{enumerate}
\item[($\ast$)]\label{*} The set of relations $\mathcal{R}$ is equal to the set of all reduced words of length not greater than $3$ in $X \cup X^{-1}$ that represent $1$ in $G$;
\item[($\ast\ast$)]\label{**} For all $n\in \NN$, we have $\delta_G (n)>0$.
\end{enumerate}
\end{defn}
Condition ($\ast$) will be necessary to apply a result of Papasoglu (see Lemma \ref{m>r} below) in the proof of the upper bound in (\ref{fpspectrum}). Condition ($\ast\ast$) is 
equivalent to $\delta_G(1)>0$, and is also equivalent to the existence of a generator that represents 1 in $G$; it is necessary to guarantee that the lower bound in (\ref{fpspectrum}) is well-defined. Note that these conditions are non-restrictive, as the following well-known lemma shows. We provide a short proof for the convenience of the reader.

\begin{lem}
Every finitely presented group $G$ has a finite presentation $G=\langle X\mid \mathcal R\rangle$ satisfying conditions ($\ast$) and ($\ast\ast$) from Definition \ref{Gpres}.
\end{lem}

\begin{proof}
The first condition can be ensured by the usual triangulation procedure. Namely, we start with any finite presentation $G=\langle X_0\mid \mathcal R_0\rangle $. For every relator $x_1x_2\ldots x_k\in \mathcal R_0$, where $k>3$ and $x_i\in X_0^{\pm 1}$ for all $i$, we add extra generators $y_1, \ldots, y_{k-1}$ and relations $$y_1=x_1x_2,\;\;\; y_2=y_1x_3,\;\;\; \ldots,\;\;\; y_{k-1}=y_{k-2}x_k.$$ After that, we replace the relation $x_1x_2\ldots x_k=1$ with $y_{k-1}=1$. Finally, we add all relations of the form $R=1$, where $R$ is a reduced word in the alphabet $X\cup X^{-1}$ of length at most $3$ representing $1$ in $G$. The resulting presentation satisfies ($\ast$). Since it is obtained from $\langle X_0\mid \mathcal R_0\rangle$ by a finite sequence of Tietze transformations, it represents the same group $G$. The condition  ($\ast\ast$) can always be ensured by adding a redundant generator $z$ and the relation $z=1$.
\end{proof}

\begin{defn} For a van Kampen diagram $\Delta$ over $G=\langle X\mid \mathcal R\rangle$, we denote by $\d_\Delta $ the graph metric on its $1$-skeleton. The \emph{radius} of $\Delta$ is defined by the formula
$$
r(\Delta)=\max\{ \d_\Delta  (u, \partial \Delta)\mid u\; {\rm is \;a \;vertex \;of\;} \Delta\}.
$$
The \emph{perimeter} of $\Delta$ is $\ell(\partial \Delta)$.
\end{defn}
In this notation, the following result was proved in \cite[Section 2]{Pap}.

\begin{lem}\label{m>r}
Let $\Delta$ be a van Kampen diagram over $\langle X\mid \mathcal R\rangle$, satisfying condition ($\ast$) from Definition \ref{Gpres}. For every $m\ge r(\Delta)$, we can cut $\Delta$ into at most $\ell(\partial \Delta)/m +1$ subdiagrams of perimeter at most $25m$.
\end{lem}

We are now ready to prove the theorem that connects the Dehn spectrum to the Dehn function of finitely presented groups. 

\begin{proof}[Proof of Theorem \ref{Thm:FP}]
Let $\langle X\mid \mathcal R\rangle$ be a finite presentation of $G$ satisfying ($\ast$) and ($\ast\ast$) from Definition \ref{Gpres} and let $\delta_G(n)$ denote its Dehn function. Throughout the proof, we use the notation $\mathcal S_k$ and $\Area_{\mathcal S_k}$ introduced in Definition \ref{Def:main}. Note that $\mathcal S_3=\mathcal R$ by ($\ast$).

We begin by proving the upper bound in (\ref{fpspectrum}). For every word $w$ in the alphabet $X^{\pm 1}$ representing $1$ in $G$ and any positive integer $m$, we will show that
\begin{equation}\label{A25m}
\Area_{\mathcal S_{25m}}(w) \le \theta(w)/m +1,
\end{equation}
where
$$
\theta(w)=15 \Area_{\mathcal S_3}(w) +\| w\|.
$$
This implies $f_{G,X}(k, 25m, n) \le (15\delta_G(n)+n)/m +1$ and the desired upper bound follows.

The proof of (\ref{A25m}) is by induction on $\theta(w)$. If $\theta (w)\le 25 m$ the claim is obvious as $\| w\|\le 25m$ and $\Area_{\mathcal S_{25m}}(w)\le 1$ in this case. Assume that $\theta(w)>25 m$.

Throughout the rest of the proof, by the \emph{area} of a van Kampen diagram $\Delta$, denoted by $\Area(\Delta)$, we mean the number of faces in $\Delta$. Let $\Delta $ be a minimal area van Kampen diagram over $\langle X\mid \mathcal R\rangle$ such that $\lab(\partial\Delta )\equiv w$.  There are two cases to consider.

{\it Case 1.} Suppose first that $m\ge r(\Delta)$. Applying Lemma \ref{m>r} to $\Delta$, we obtain $$\Area_{\mathcal S_{25m}}(w) \le \ell(\partial \Delta)/m +1=\|w\|/m +1\le \theta(w)/m +1.$$

\begin{figure}
  \begin{center}
 \hspace{18mm} 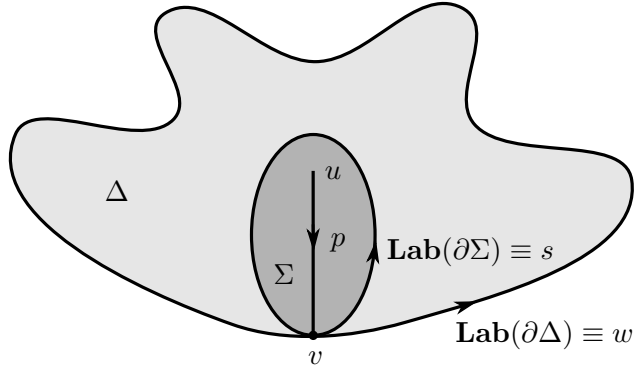
  \end{center}
  \vspace{-3mm}
  \caption{Case 2 in the proof of the upper bound in Theorem \ref{Thm:FP}}\label{fig4}
\end{figure}

{\it Case 2.} Now assume that $m<r(\Delta)$. Let $u$ be a vertex of $\Delta$ such that $\d_\Delta (u, \partial \Delta)=m$ and let $p$ denote a geodesic path in $\Delta^{(1)}$ connecting $u$ to the closest vertex $v\in \partial \Delta$ (see Fig. \ref{fig4}). Clearly, $\ell(p)= m$. Let $\Sigma $ be the minimal simply connected subdiagram of $\Delta $ containing $p$ and all faces of $\Delta$ that have at least one common edge with $p$. Since $p\cap \partial \Delta =\{v\}$, every edge of $p$ belongs to a face of $\Delta$. By ($\ast$), we have $\Area(\Sigma)\ge m/3$ (we do not care about optimal constants here). On the other hand, every face $\Pi$ of $\Sigma$ that shares an edge with $\partial \Sigma$ must also share an edge with $p$, as otherwise $\Pi$ can be excluded from $\Sigma$. This implies that $\ell(\partial \Sigma)\le 4m$.

Passing to a cyclic shift of $w$ if necessary, we can assume that $w$ is read along $\partial \Delta$ starting from the vertex $v$ in the counterclockwise direction. Let  $s\equiv\lab(\partial \Sigma)$, where the label is also read counterclockwise starting from $v$ (see Fig. \ref{fig4}). By cutting off $\Sigma $, we get the decomposition $w=w^\prime  s$ in the free group $F(X)$, where $w^\prime =ws^{-1}$. We have
$$
\Area _{\mathcal S_3}(w^\prime) \le \Area (\Delta) - \Area(\Sigma)\le \Area_{\mathcal S_3}(w)-m/3,
$$
and $$\| w^\prime\|\le \|w\| + \| s\| \le \| w\|+ 4m.$$
Hence,
\begin{equation}\label{tW'}
\theta (w^\prime)\le  15 (\Area_{\mathcal S_3} (w)-m/3) + \| w\| +4m  = \theta (w)- m.
\end{equation}
Using the inductive assumption, the inequality $\| s\| \le 4m < 25 m$, and (\ref{tW'}), we obtain
$$
\Area_{\mathcal S_{25m}}(w)\le \Area_{\mathcal S_{25m}}(w^\prime) +1 \le \theta (w^\prime)/m +2 \le \theta (w)/m +1.
$$

Next, we prove the lower bound in (\ref{fpspectrum}).
Fix any $m\ge 3$. By the definition of the Dehn spectrum, every word $w\in \ll \mathcal R\rr$ of length $\| w\| = n$  decomposes as
\begin{equation}\label{Eq:w=}
w=\prod_{i=1}^\ell f_i^{-1}R_i f_i,
\end{equation}
where $f_i \in F(X)$, $R_i\in \mathcal S_m$, and $\ell \le f_{G,X}(3,m,n)$. Further, every $R_i$ decomposes as a product of at most $\delta_G(m)$ conjugates of elements of $\mathcal R$ in $F(X)$ by the definition of the Dehn function. Substituting these decompositions in (\ref{Eq:w=}), we obtain the inequality $\Area_{\mathcal R} (w)\le \ell \delta_G(m)$. Therefore, $\delta_G(n)\le f_{G,X}(3,m,n)\delta_G(m)$ for all $n$ and all $m\ge 3$ and the left inequality in (\ref{fpspectrum}) follows.
\end{proof}

We now discuss the dependence of the lower and upper bounds in (\ref{fpspectrum}) on the choice of the finite presentation of $G$. To this end, we recall the definition of the equivalence relation commonly used for Dehn functions.

\begin{defn}
Let $f, g\colon \NN\to \NN\cup \{ 0\}$. We write $f\precapprox g$ if there exists $C\in \NN$ such that $f(n)\le Cg(Cn)+Cn$. Note that $\precapprox$ is a quasi-order on the set of all non-decreasing functions $\NN\to \NN$. By $\approx$ we denote the induced equivalence relation.
\end{defn}

It is well-known that the Dehn functions of any two finite presentations of the same group are $\approx$-equivalent.

\begin{defn}
We say that a function $f\colon \NN\to \NN$ has \emph{at least linear growth} if $\inf_{n\in \NN} f(n)/n>0$.
\end{defn}

The following lemma is likely known to experts. However, it does not seem to have been recorded in the literature. The closest result we could find is a considerably weaker theorem proved in \cite{Bat}: if a group $G$ has a finite presentation with Dehn function satisfying $\lim_{n\to \infty}\delta_G(n)/n=0$, then $G$ is either free or finite.

\begin{lem}\label{Lem:sl}
Let $G=\langle X\mid \mathcal R\rangle$ be a finite presentation. If $\ll \mathcal R\rr\ne \{ 1\}$ in $F(X)$, then the Dehn function of $G$ with respect to this presentation has at least linear growth. More precisely, for every $R_0 \in \mathcal{R}$ such that $R_0\ne 1$ in $F(X)$, the area of $R_0^n$ has at least linear growth with respect to $n$.
\end{lem}

\begin{proof}
Recall that a map $\phi\colon F(X)\to \mathbb R$ is a \emph{homogeneous quasi-character} if there exists a constant $K$ such that $\phi(f^n)=n\phi(f)$ and $|\phi(fg)-\phi(f)-\phi(g)|\le K$ for all $f,g\in F(X)$. It is well known that for any non-trivial element $f\in F(X)$, there is a homogeneous quasi-character $\phi\colon F(X)\to \mathbb R$ such that $\phi (f)>0$ (see the beginning of Section 3 in \cite{Bro} or \cite{BF} for more details and greater generality).

Let $\phi\colon F(X)\to \mathbb R$ be a homogeneous quasi-character such that $\phi (R_0)>0$. The definition of a homogeneous quasi-character implies that for every $f_i \in F(X)$, $R_i\in \mathcal R$, and $\e_i=\pm 1$, we have
\begin{equation}
\begin{split}
\left|\phi\left(\prod\limits_{i=1}^\ell f_i^{-1}R_i^{\e_i} f_i\right)\right| & \le \sum\limits_{i=1}^\ell \left|\phi (f_i^{-1}R_i^{\e_i}f_i)\right| + (\ell-1)K\\ & \le\sum\limits_{i=1}^\ell(|\phi(R_i)|+2K)+(\ell-1)K\le \ell(M+3K),
\end{split}
\end{equation}
where $M=\max_{R\in \mathcal R}|\phi (R)|$. On the other hand, $\phi(R_0^n)=n\phi(R_0)$. This implies that $\Area_{\mathcal R} (R_0^n)\ge n|\phi(R_0)|/(M+3K)$ and, therefore, the Dehn function of $\langle X\mid \mathcal R\rangle$ has at least linear growth.
\end{proof}

\begin{lem}\label{Lem:ab}
Let $\alpha_1,\alpha_2,\beta_1,\beta_2\colon \NN\to \NN$ be any non-decreasing functions such that $\alpha_1 \approx\alpha_2$ and $\beta_1\approx\beta_2$. Suppose that $\beta_1$ and $\beta_2$ have at least linear growth. Then $\alpha_1(n)/\beta_1(m) \sim \alpha_2(n)/\beta_2(m)$, where both sides are considered as functions of triples $(k,m,n)$ independent of $k$.
\end{lem}

\begin{proof}
The assumptions of the lemma imply that we can find $C\in \NN$ such that $\alpha _1(n)\le C\alpha_2(Cn)+Cn$ and $\beta _2(n)\le C\beta_1(Cn)$ for all $n\in \NN$. In turn, these inequalities imply
$$
\frac{\alpha_1(n)}{\beta_1(Cm)} \le \frac{C\alpha_2(Cn)+Cn}{\beta_2(m)/C} \le  C^2 \frac{\alpha_2(Cn)}{\beta_2(m)} + \frac{C^2}{\inf_{m\in \NN}\beta_2(m)/m}\cdot \frac nm .
$$
Therefore, $\alpha_1(n)/\beta_1(m) \preccurlyeq \alpha_2(n)/\beta_2(m)$. Similarly, $\alpha_2(n)/\beta_2(m) \preccurlyeq \alpha_1(n)/\beta_1(m)$.
\end{proof}

\begin{rem}
The assumption that $\beta_1$ and $\beta_2$ have at least linear growth cannot be omitted. Indeed, $1\approx m$ but $n\not\sim n/m$.
\end{rem}

The following corollary justifies Remark \ref{Rem:pres}.

\begin{cor}\label{Cor:fp}
Let $G$ be a finitely presented group. Suppose that $\delta_G$ is the Dehn function of a finite presentation of $G$ such that $\delta_G(m)>0$ for all $m\in \NN$. Then the inequalities (\ref{fpspectrum}) of Theorem \ref{Thm:FP} hold.
\end{cor}

\begin{proof}
Note that if $G=\langle Y\mid \mathcal S\rangle$ is a finite presentation of $G$ and the corresponding Dehn function takes at least one positive value, then $\ll \mathcal S\rr\ne \{ 1\}$ in $F(Y)$. Therefore, the Dehn function of $G$ with respect to $\langle Y\mid \mathcal S\rangle$ has at least linear growth by Lemma \ref{Lem:sl} (in particular, this is true for the finite presentation $\langle X\mid \mathcal R\rangle$ used in the proof of Theorem \ref{Thm:FP}, because we assumed that condition $(\ast\ast)$ from Definition \ref{Gpres} holds). Thus, Lemma \ref{Lem:ab} applied to the lower and upper bounds in Theorem \ref{Thm:FP} allows us to replace the Dehn function with respect to $\langle X\mid \mathcal R\rangle$ with any other Dehn function of  $G$ that takes only positive values.
\end{proof}

\begin{cor}\label{Cor:hyp}
Let $G$ be a finitely presented group. We have $f_G(k,m,n)\sim \alpha(n)/m$ for some function $\alpha\colon \NN\to \NN\cup\{0\}$ if and only if $G$ is hyperbolic.
\end{cor}

\begin{proof}
The backwards implication follows immediately from Theorem \ref{Thm:FP} and linearity of the Dehn function of hyperbolic groups. Let us prove the forward implication. We fix a finite generating set $X$ of $G$. By our assumption, there exists $C\in \NN$ such that
\begin{equation}\label{Eq:amf1}
\frac{\alpha(n)}{Cm} \le Cf_{G,X} (Ck, m, Cn)+ C\frac nm +C
\end{equation}
for all $(k,m,n)\in \NN^3$ satisfying $m\ge Ck$. Substituting $k=1$, $m=Cn$ in (\ref{Eq:amf1}) and taking into account the inequality $f_{G,X} (C, Cn, Cn)\le 1$, we obtain that $\sup_{n\in \NN}\alpha(n)/n<\infty$. This and the inequality $f_G(k,m,n)\preccurlyeq \alpha(n)/m$ imply the existence of a constant $D$ such that
\begin{equation}\label{Eq:amf2}
f_{G,X}(k,Dm,n)\le D n/m +D
\end{equation}
for all $m\ge D k$.

Since $G$ is finitely presented, for any choice of a finite set of relators, there exists $M$ such that the Dehn function satisfies $\delta_G(n)=f_{G, X}(M, M, n)$ for all sufficiently large $n$ (see Remark \ref{Ex:fp}). Combining the latter equality with (\ref{Eq:amf2}), we conclude that $G$ has a linear Dehn function and, therefore, it is hyperbolic.
\end{proof}

Corollary \ref{Cor:hyp} can be used to provide examples of functions that \emph{cannot} be equivalent to the Dehn spectrum of any finitely presented group. E.g., the function $n^2/m$ has this property.

Now we prove another corollary of Theorem \ref{Thm:FP}.

\begin{proof}[Proof of Corollary \ref{Cor:Quad}]
Let $G$ be a group given by a finite presentation $\langle X\mid \mathcal R\rangle$ satisfying ($\ast$) from Definition \ref{Gpres}. Suppose that the corresponding Dehn function satisfies $\delta_G(n)\le Cn^2$ for some $C\in \NN$. Papasoglu proved that there exists a constant $K\in \NN$ such that for any $a\in \mathbb N$, one can cut a minimal area van Kampen diagram over $\langle X\mid \mathcal R\rangle$ with boundary length $n>100 a$ into less than $Ka^2$ subdiagrams of perimeter at most $n/a$ each (see the first paragraph on p. 803 in \cite{Pap}).

Fix any $(k,m,n)\in \NN^3$ such that $n\ge m>100$ and let $a= \lfloor n/m\rfloor$. Thus, we have $2m>n/a \ge m>100$. Let $w$ be any word in $X\cup X^{-1}$ representing $1$ in $G$ of length $\| w\| = n$ and let $\Delta $ be a minimal area van Kampen diagram over the presentation $\langle X\mid \mathcal R\rangle$  such that $\Lab (\partial \Delta)\equiv w$. Applying the result of Papasoglu to $\Delta $, we obtain
$$
\Area_{\mathcal S_{2m}}(w)\le Ka^2 \le K {n^2}/{m^2}.
$$
Thus, $f_{G,X} (k,2m,n) \le K(n/m)^2$ for all $(k,m,n)\in \NN^3$ such that $n\ge m>100$. Note also that $f_{G,X}(k,m,n)\le 1$ whenever $n\le m$. Therefore, we have $f_{G,X}(k,m,n)\preccurlyeq (n/m)^2$. Combining this inequality with the lower bound in (\ref{fpspectrum}), we obtain $f_{G,X}(k,m,n)\sim (n/m)^2$.
\end{proof}

We conclude this section with the proof of Proposition \ref{Prop:AC}. This result is simply a reformulation of known ones, and so we keep this part of our paper brief. For the precise definition and basic properties of asymptotic cones, we refer the reader to \cite{DK}. For the purpose of this paper, the property ``all asymptotic cones are simply connected" can be taken as a black box thanks to the following reformulation of the equivalent loop division property \cite[Theorem 4.4]{Dru} (see also \cite{Gro93, Pap}).

\begin{lem}\label{Lem:Dru}
Let $\langle X\mid \mathcal R\rangle $ be a finite presentation of a group $G$, $K=\max\{ \| R\| \mid R\in \mathcal R\}$. All asymptotic cones of $G$ are simply connected if and only if there exist $\nu, M, \ell\in \NN $ such that $M\ge 2$ and for every $n\ge \ell$, we have
\begin{equation}\label{Eq:LDP}
 f_G(K,\lfloor n/M \rfloor, n) \le \nu.
\end{equation}
\end{lem}

\begin{proof}[Proof of Proposition \ref{Prop:AC}]
The implication (c) $\Rightarrow$ (b) is obvious, so we only need to prove (b) $\Rightarrow$ (a) $\Rightarrow$ (c).

Assume (b) holds, i.e., there exists $C\in \NN$ such that $f_{G,X}(K, Cm, n)\le Cg(Cn/m)+Cn/m +C$ for all $m\ge CK$. This implies that there are $\nu, \ell\in \NN$ such that $f_{G,X}(K, \lfloor n/2\rfloor , n)\le \nu$ for every $n\ge \ell$. Applying Lemma \ref{Lem:Dru}, we obtain (a).

Finally, assume that (a) holds. Let $\nu, M, \ell\in \NN $ be the constants provided by Lemma~\ref{Lem:Dru}. Fix any $m>\ell$ and let $\mu = \log_M n/m $. We have $\lfloor n/M^{\mu -1}\rfloor  \ge m>\ell$. Therefore, we can iterate the inequality (\ref{Eq:LDP}) $\lfloor\mu \rfloor$ times; this yields the inequality $f_{G,X}(k, m, n) \preccurlyeq \nu ^{\log_M n/m} = (n/m)^ {\log _M \nu}$.
\end{proof}

As we already mentioned in the introduction, Proposition \ref{Prop:AC} can be used to exhibit examples of finitely presented groups for which the left inequality in (\ref{fpspectrum}) is strict. For example, let $G$ be the finitely presented group with Dehn function $\delta_G(n)\approx n^5$ and at least one non-simply connected asymptotic cone constructed in \cite[Theorem 3.1]{Bri}. The equivalence $f_G(k,m,n)\sim \delta_G(n)/\delta_G(m)$ would imply $f_G(k,m,n) \sim (n/m)^5$ by Lemma \ref{Lem:ab}, which contradicts Proposition \ref{Prop:AC}. Moreover, we have the following.

\begin{cor}\label{Cor:H3}
Let $$G=\langle a,b, s,t\mid s^{-1}as=a,\, t^{-1}at=a,\, s^{-1}bs=ba,\, t^{-1}bt=ba \rangle,$$
$$ H=\langle a,b,c \mid [a,b]=c, \, [a,c]=[b,c]=1\rangle .$$ The groups $G$ and $H$ have a cubic Dehn function but $f_G(k,m,n)\not \sim f_H(k,m,n)$.
\end{cor}

\begin{proof}
Olshanskii and Sapir showed that $\delta _G(n)$ is cubic and has a non-simply connected asymptotic cone \cite{OS}. On the other hand, it is well-known that the $3$-dimensional Heisenberg group $H_3$ has cubic Dehn function (see, for example, \cite{Ger}) and simply connected asymptotic cones (see, for example, \cite{Pan}). By Proposition \ref{Prop:AC}, the simple connectivity of all asymptotic cones of a finitely presented group can be detected by using the equivalence class of the Dehn spectrum. Therefore, $f_G(k,m,n)$ cannot be equivalent to $f_H(k,m,n)$.
\end{proof}

\begin{rem} The standard method of computing the Dehn function of $H$ via combing suggests that $f_H(k,m,n)\sim (n/m)^3$. We leave it to the interested reader to verify the details. It would also be nice to compute exactly the isoperimetric spectra of other finitely generated nilpotent groups, and the Olshanskii-Sapir group $G$. 
\end{rem}


\section{Groups with linear Dehn spectrum}\label{Sec:GwLS}



\subsection{Limits of hyperbolic groups}

Recall that a geodesic metric space $S$ is said to be $\delta$-hyperbolic for some $\delta\ge 0$ if every geodesic triangle in $S$ is $\delta$-thin; that is, every side of the triangle is contained in the union of the closed $\delta$-neighborhoods of the other two sides. The metric space is said to be {\it hyperbolic} if it is $\delta$-hyperbolic for some $\delta \geq 0$. A group $G$ is \emph{hyperbolic} if its Cayley graph $\Cay(G,X)$ with respect to some (equivalently, any) finite generating set $X$ is hyperbolic.

In this section, we study the Dehn spectrum of groups that can be approximated by hyperbolic ones in a certain sense. We begin by recalling an analog of the Dehn algorithm for hyperbolic graphs.

\begin{lem}[]\label{Lem:Dhyp}
Let $\Gamma $ be a $\delta $-hyperbolic graph for some $\delta\in \NN$. For every loop $c$ in $\Gamma$, the following hold.
\begin{enumerate}
\item[(a)] There exists a subpath $p$ of $c$  such that $\d (p_-, p_+)\le \ell (p)-1$ and $\d (p_-, p_+)+ \ell (p)\le 16\delta$.
\item[(b)] The loop $c$ admits a combinatorial $16\delta$-filling of area at most $\ell(c)$.
\end{enumerate}
\end{lem}
\begin{proof}
 The proof of (a) can be found in {\cite[Chapter III.H, Lemma 2.6]{BH}}; (b) follows from (a) by induction.   
\end{proof}

Let $\Gamma $ be a graph and let $s\in \NN$. The \emph{$s$-expansion} of $\Gamma$, denoted by $\Gamma^s$, is the graph obtained from $\Gamma$ by adding edges between every pair of vertices $a,b\in V(\Gamma)$ such that $1<\d_\Gamma (a,b)\le s$, where $\d_\Gamma$ denotes the metric on $\Gamma$.

\begin{lem}[{\cite[Lemma 4.2]{MO}}]\label{Lem:MO}
Let $\Gamma $ be a $\delta$-hyperbolic graph. For every $s\in \NN$, the graph $\Gamma^s$ is $(\delta/s+6)$-hyperbolic.
\end{lem}

\begin{cor}\label{Cor:fhyp}
Let $\delta \in \NN$ and let $\Gamma $ be a $\delta$-hyperbolic graph. For every integer $m\ge 200\delta$, every loop $c$ in $\Gamma$ admits a combinatorial $m$-filling of area at most $400\lceil\ell(c)/m\rceil$.
\end{cor}

\begin{proof}
Let $s = \lfloor(m-16\delta)/100\rfloor$. Note that $s>0$ for every $m\ge 200 \delta$. We think of $\Gamma $ as a subgraph of the $s$-expansion $\Gamma^{s}$. Let $r=\lceil \ell(c)/s \rceil$ and let $c=c_1\ldots c_r$, where $\ell(c_i)\le s$ for all $i$. Connecting $(c_i)_-$ to $(c_{i})_+$ by an edge $e_i$ in $\Gamma ^{s}$, we obtain a loop $d=e_1\ldots e_r$ in $\Gamma ^{s}$. By Lemma~\ref{Lem:MO}, $\Gamma ^{s}$ is $(\delta/s +6)$-hyperbolic. Therefore,  $\Area_{16\delta/s +96} (d)\le \ell(d)=r$ by Lemma~\ref{Lem:Dhyp}.

Let $\alpha\colon V(\Gamma^s)\to V(\Gamma)$ denote the identity map and let $\tilde c=\tilde c_1 \cdots \tilde c_r$ be a coarse $\alpha$-image of $d$ (see Definition~\ref{Def:im}); since every $\tilde c_i$ is geodesic, we have 
\begin{equation}\label{Eq:lcit}
\ell(\tilde c_i)\le \ell(c_i)\le s,
\end{equation}
for all $i=1, \ldots, r$.  Note that $\alpha$ is $(s,0)$-Lipschitz. By Lemma \ref{Lem:indf}, $\tilde c$ admits a combinatorial $s(16\delta/s +96)$-filling $\tilde \phi\colon \tilde D\to \Gamma$ of area at most $r$. Since 
$$
s(16\delta/s+96)=16\delta +96s =16\delta +96 \lfloor(m-16\delta)/100\rfloor < m,
$$
we conclude that $\tilde \phi $ is an $m$-filling. 

\begin{figure}\label{Fig:FigD}
  \begin{center}
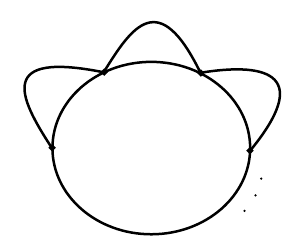
  \end{center}
  \vspace{-3mm}
  \caption{Modifying $\tilde D$ }
\end{figure}

We now transform this filling to a combinatorial $m$-filling $\phi\colon D\to \Gamma$ of $c$ by using a simplified version of the same trick as in the proof of Theorem \ref{Thm:QIG}. More precisely, we first construct $D$ by gluing additional faces $K_i$ ($i=1, \ldots , r$) along the boundary of $\tilde D$ as on Fig.~\ref{Fig:FigD}, where $\partial K_i=a_i^{-1}b_i$, and $a_i$ is the segment of $\partial \tilde D$ mapped to $\tilde c_i$ by $\tilde \phi$. Further, we let $\phi$ be the extension of $\tilde \phi$ that maps $b_i$ to $c_i$ for each $i=1, \cdots r$. For each of the additional faces, we have $\ell(\phi(\partial K_i))\le \ell (c_i)+\ell(\tilde c_i) \le 2s\le m$ by (\ref{Eq:lcit}). Since $\tilde \phi$ was an $m$-filling of area at most $r$, we obtain that $\phi $ is also an $m$-filling of area at most
$$
2r=2\left\lceil \frac{\ell(c)}{s}\right\rceil=2\left\lceil \frac{\ell(c)}{\lfloor(m-16\delta)/100\rfloor}\right\rceil < 400\left\lceil \frac{\ell(c)}{m}\right\rceil. \qedhere
$$
\end{proof}

We now introduce auxiliary terminology.

\begin{defn}\label{Def:wah}
We say that a group $G$ is \emph{approximated by hyperbolic groups} if $G$ is the direct limit of a sequence
$$
G_1\stackrel{\e_1}\longrightarrow G_2 \stackrel{\e_2}\longrightarrow \ldots,
$$
where each group $G_i$ is hyperbolic and each homomorphism $\e_i$ is surjective. Equivalently, $G$ admits a presentation
\begin{equation}\label{Eq:Glim}
\left\langle X\; \left| \;\, \bigcup_{i=1}^\infty \mathcal R_i \right.\right\rangle,
\end{equation}
where $|X|<\infty $, $\mathcal R_1\subseteq \mathcal R_2\subseteq \ldots$,  and the groups $G_i=\langle X\mid \mathcal R_i\rangle$ are hyperbolic for all $i\in \NN$. Furthermore, we say that $G$ is \emph{well-approximated by hyperbolic groups} if, in addition, the above sequence can be chosen so that there exists a constant $C\in \NN$ such that the following conditions hold for all $i\in \NN$.
\begin{enumerate}
\item[(a)] The normal closure of $\mathcal R_i$ in $F(X)$ contains all words $w\in F(X)$ of length $\| w\|\le i$ representing the identity in $G$.
\item[(b)] The Cayley graph $\Cay(G_i,X)$ of the group $G_i=\langle X\mid \mathcal R_i\rangle$ is $Ci$-hyperbolic.
\end{enumerate}
\end{defn}

Conditions introduced above have been studied in group theory (sometimes implicitly) since the 1980s.

\begin{ex}
All lacunary hyperbolic groups introduced by Olshanskii, Osin, and Sapir are approximated by hyperbolic groups (see \cite[Theorem 1.1(3)]{OOS}). 
\end{ex}

Below, we show that the class of groups well-approximated by hyperbolic groups include finitely generated $C'(1/6)$ groups (see Corollary \ref{Cor:C16}), wreath products $K \wr \ZZ$, where a group $K$ is finite (see Theorem \ref{Thm:WP}), and free Burnside groups $B(m,n)$ of sufficiently large exponent $n$ (see Theorem \ref{Thm:Bmn}).

The main result of this section is the following proposition. We note that part (b) does not generalize to all groups approximated by hyperbolic groups (see Remark \ref{Rem:appr}).

\begin{prop}\label{Prop:LimHyp}
Suppose that $G$ is approximated by hyperbolic groups.
\begin{enumerate}
\item[(a)] For every $k\in \NN$, there exists $M=M(k)\in \NN$ such that $f_{G,X}(k,m,n)\le 400\lceil n/m\rceil $ for all $m\ge M$.
\item[(b)] If $G$ is well-approximated by hyperbolic groups, then $f_{G,X}(k,m,n)\sim n/m$.
\end{enumerate}
\end{prop}

\begin{proof}
Let (\ref{Eq:Glim}) be a presentation of $G$ as in Definition \ref{Def:wah}. Fix any $k\in \NN$. There exists $i(k)$ such that $\S_k\subseteq \ll \mathcal R_{i(k)}\rr $, where $\S_k=\{w\in N_X \mid \| w\|\le k\}$, as defined by (\ref{Sk}). Let $M(k)=200 \delta_k$, where $\delta_k$ is the hyperbolicity constant of the Cayley graph $\Cay(G_{i(k)}, X)$. Note that every $w\in \ll \mathcal S_k\rr$ labels a loop in $\Cay(G_{i(k)}, X)$. Applying Corollary \ref{Cor:fhyp} to $\Cay(G_{i(k)}, X)$, we obtain
\begin{equation}\label{Eq:Aw}
\Area_{\mathcal S_m}(w) \le 400\left\lceil \| w\|/m\right\rceil.
\end{equation}
for all $m\ge M(k)$ and claim (a) follows.

If $G$ is well-approximated by hyperbolic groups, we can assume that $i(k)=k$ and $M(k)\le 200Ck$, where $C$ is the constant from Definition \ref{Def:wah}. Thus, (\ref{Eq:Aw}) holds for all $m\ge 200Ck$ and we obtain $f_{G,X}(k,m,n)\sim n/m$.
\end{proof}

Recall that a function $f(k,m,n)\in \mathcal F$ is essentially independent of $k$ if it is equivalent to a function $g\in \mathcal F$ that depends on $m$ and $n$ only.

\begin{cor}\label{Cor:LHk}
Suppose that $G$ is approximated by hyperbolic groups. If $f_G(k,m,n)$ is essentially independent of $k$, then $f_G(k,m,n)\sim n/m$.
\end{cor}

\begin{proof}
Recall that $\mathbb T=\{ (k,m,n)\in \mathbb N\times \mathbb N\times\mathbb N\mid m\ge k \}$. Suppose that $f_G(k,m,n)\sim g(m,n)$. By definition, there exists a finite generating set $X$ of $G$ and a constant $C\in \NN$ such that $g(Cm,n)\le f_{G,X}(Ck,m,Cn) +Cn/m +C$ for all $(k,m,n)\in \mathbb T$ satisfying the inequality $m\ge Ck$. Let $k=1$ and let $M$ be the constant provided by Proposition \ref{Prop:LimHyp}. For every $m\ge \max\{ M, C\}$, we have
$$
g(Cm,n)\le f_{G,X}(C,m,Cn) +Cn/m +C\le 400\left\lceil \frac{Cn}{m}\right\rceil +C\frac{n}{m} +C \preceq n/m
$$
and the result follows.
\end{proof}

We conclude this section with the first examples of infinitely presented groups with a linear Dehn spectrum.
Recall that a word $w$ in the alphabet $X\cup X^{-1}=\{x_1, x_1^{-1}, x_2, x_2^{-1}, \ldots \} $ is \emph{reduced} if it contains no subwords of the form $x_ix_i^{-1}$ and $x_i^{-1}x_i$. A group presentation $\langle X\mid \mathcal R\rangle $
is said to be \emph{reduced} if all cyclic shifts of every word $R\in \mathcal R$ are reduced.

\begin{defn}\label{SCC}
A presentation $\langle X\mid \mathcal R\rangle $ satisfies the $C^\prime(\lambda)$ \emph{small cancellation condition} for some $\lambda \in [0, 1]$ if for any cyclic shifts $R_1\not\equiv R_2$ of some relators from $\mathcal R$ or their inverses, any common initial subword $U$ of $R_1$ and $R_2$ has length
$$
\| U\| < \lambda \min\{ \| R_1\|, \, \| R_2\|\}.
$$
\end{defn}

The result below was originally proved by Greendlinger \cite{Green} using combinatorial arguments. For a contemporary geometric proof, see \cite[Section 12]{book}.

\begin{lem}[Greendlinger lemma]\label{Lem:Green}
Suppose that a group $G$ has a reduced presentation $\langle X\mid \mathcal R\rangle $ satisfying $C^\prime (1/6)$. Then every non-empty reduced word in the alphabet $X\cup X^{-1}$ that represents $1$ in $G$ contains a subword $U$ such that $U$ is also a subword of a cyclic shift of some $R\in \mathcal R^{\pm 1}$ and $\| U\| > \|R\|/2$.
\end{lem}

\begin{cor}\label{Cor:C16}
Let $G$ be a finitely generated group admitting a (possibly infinite) presentation satisfying the $C^\prime(1/6)$ small cancellation condition. Then the Dehn spectrum of $G$ is linear.
\end{cor}

\begin{proof}
Let $G=\langle X\mid \mathcal R\rangle $ be a presentation satisfying the $C^\prime(1/6)$ small cancellation condition. Let $\mathcal R_i=\{ R\in \mathcal R\mid \| R\| \le 2i\}$. Iterated application of Lemma \ref{Lem:Green} implies that every word $w\in F(X)$ of length $\|w\|\le i$ representing $1$ in $G$ belongs to $\ll \mathcal R_i\rr$. It is also known that if $H=\langle Y\mid \mathcal S\rangle$ is a finite presentation satisfying $C^\prime(1/6)$, then $\Cay(H,Y)$ is $\delta$-hyperbolic for $\delta =\max\{ \|S\| \mid S\in \mathcal S\}$ (see Theorem 36 in \cite{Str}). Thus $G$ is well-approximated by hyperbolic groups and part (b) of Proposition \ref{Prop:LimHyp} applies.
\end{proof}

\begin{rem}
We expect the conclusion of the last corollary to hold well beyond the classical $C'(1/6)$ settings. For example, the same proof relying on \cite{Str} works for group presentations satisfying the $C'(1/4) \,\&\, T(4)$ condition. Furthermore, we expect the result to remain valid in appropriate graphical small cancellation settings. For purely non-metric conditions such as $C(7)$, the situation is less clear and merits further investigation. We neither define nor discuss these small cancellation conditions in our paper and refer the interested reader to \cite{Gru,Str} for more details.
\end{rem}


\subsection{Wreath products}

The main goal of this section is to prove part (a) of Theorem \ref{Thm:Lin}, which states that every wreath product $K \wr \mathbb{Z}$, with $K$ finite, has a linear Dehn spectrum. We begin by estimating hyperbolicity constants of groups acting on trees.

Let $G$ be a group generated by a set $X$. For every finite subset $S\subseteq G$, we define its \emph{radius} by
$$
\Rad_X(S)=\max\{ |g|_X \mid g\in S\},
$$
where $|g|_X$ is the word length of $g$ with respect to $X$.

\begin{prop}\label{Prop:Tree}
Let $G$ be a group acting by isometries on a tree $T$. Let $v$ be a vertex of $T$ and let $X$ be a generating set of $G$ such that
\begin{equation}\label{Eq:dT}
\d_T(v, xv)\le 1\;\;\; \forall\, x\in X\cup X^{-1}.
\end{equation}
Suppose that the stabilizer of $v$ is finite. Then the Cayley graph $\Cay(G,X)$ is $\delta$-hyperbolic for $\delta = 2\Rad_X(\Stab_G(v))$.
\end{prop}

\begin{proof}
We claim that the map $G\to V(T)$ defined by $g\mapsto gv$ for all $g\in V$ extends to a generalized combinatorial map $\psi\colon \Cay(G,X)\to T$. Indeed, suppose that some $g,h\in G$ are connected by an edge $e$ in $\Cay(G,X)$. Then $h=gx$ for some $x\in X\cup X^{-1}$. Using (\ref{Eq:dT}), we obtain $$\d_T(gv, hv)=\d_T(gv, gxv)=\d_T(v, xv)\le 1, $$ i.e., the vertices $gv$ and $hv$ are connected by an edge in $T$ or coincide. In the former case, we define $\psi (e)$ to be the edge connecting $gv$ and $hv$ in $T$; in the latter case, we let $\psi(e)=gv=hv$.

Note that for every path $s$ in $\Cay(G,X)$, its image $\psi(s)$  contains the geodesic path in $T$ connecting $\psi(s)_-$ to $\psi(s)_+$. We denote this geodesic path by $\widehat s$. Let now $\Delta =pqr$ be a triangle in $\Cay (G,X)$ with geodesic sides $p$, $q$, $r$ and let $u$ be a vertex on $p$. We will show that $u$ is within distance at most $\delta$ from $q\cup r$.

\begin{figure}
  \begin{center}
 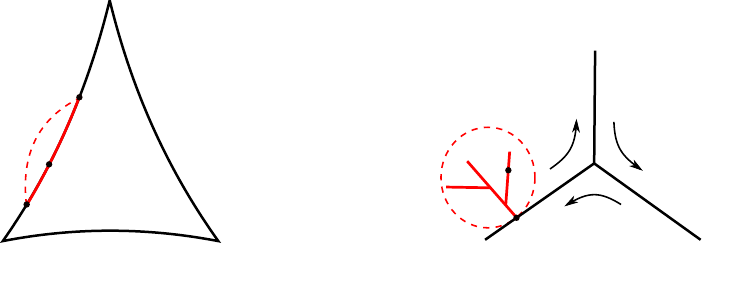
  \end{center}
  \vspace{-3mm}
  \caption{The map $\psi\colon \Cay(G,X) \to T$.}\label{Fig:6}
\end{figure}

In general, $\psi(u)$ may not lie on $\widehat p$. However, there always exists a subpath $p_u$ of $p$ containing $u$ such that $\psi(p_u)_-=\psi(p_u)_+\in \widehat p$ (if $u$ lies on $\widehat p$, $p_u$ reduces to one point). Shifting $\Delta $ by an isometry of $\Cay(G,X)$, we can assume that $(p_u)_-=1$ and, therefore, $\psi(p_u)_-=\psi(p_u)_+=v$ (see Fig. \ref{Fig:6}).

Let $g_u\in G$ be the element represented by $\Lab(p_u)$. We have $g_u\in \Stab_G(v)$. Thus, $u$ belongs to a geodesic path in $\Cay(G,X)$ connecting $1$ to an element of $\Stab_G(v)$. Hence, $|u|_X\le \Rad_X(\Stab_G(v))$. Since $T$ is a tree, we have $v\in \widehat q$ or $v\in \widehat r$. Therefore, there is a vertex $w\in q\cup r$ such that $\psi (w)=v$. Again, we have $w\in \Stab_G(v)$ and
$|w|_X\le \Rad_X(\Stab_G(v))$. Finally, we obtain $$\d_X(u,w)\le |u|_X+|w|_X \le 2\Rad_X(\Stab_G(v)). \qedhere$$
\end{proof}

Now we turn to wreath products. Let $K_0$ be a finite group. We consider the presentation
$$K_0=\la Y\mid \mathcal S \ra, $$ where $Y=K_0$ and $\mathcal S$ is the set of all words in $Y$ of length at most $3$ representing $1$ in $K_0$. For elements $x$, $y$ of some group, we abbreviate $x^{-1}yx$ by $y^x$ and $x^{-1}y^{-1}xy$ by $[x,y]$.
Given $k\in \NN$, let 
$$
X=Y\cup\{ t\},
$$

$$
\mathcal R_k=\mathcal S\cup \{ \, [a^{t^i}, b^{t^{j}}]\mid i,j\in \ZZ, \; i\ne j,\;  |i|, |j|\le k,\; a, b \in Y \,\},
$$
and 
$$
G_k=\langle X\mid \mathcal R_k\rangle .
$$
Define 
$$
\mathcal{R}=\bigcup_{k\in \NN} \mathcal{R}_k.
$$
It is well-known (see, for example, \cite[Section 15]{Joh}) that the wreath product $G=K_0\wr \ZZ $ has the presentation
$G=\langle X\mid \mathcal{R} \rangle$ in the above notation. 

We view $K_0$ as the subgroup of $G$ generated by $Y$ and let $K_i= t^{-i}K_0t^i$. In this notation, we have
$$
G=\left(\bigoplus_{i\in \ZZ} K_i\right) \rtimes \ZZ,
$$
where $\ZZ=\langle t\rangle$ acts on $\bigoplus_{i\in \ZZ} K_i$ by shifts.

\begin{lem}\label{Lem:k}
Every word $w\in F(X)$ representing $1$ in $G$ represents $1$ in $G_{\|w\|}$.
\end{lem}
\begin{proof}
Let $w\in F(X)$ be a word of length $k$ representing $1$ in $G$. Using relations from the set $\mathcal S$, we can rewrite $w$ as
\begin{equation}\label{w'}
w^\prime \equiv t^{\tau_0}a_1 t^{\tau_1}\ldots a_qt^{\tau_{q}},
\end{equation}
where $q$ is a non-negative integer (we include the possibility $q=0$ to encompass the case $w^\prime =1$ in $F(X)$), $a_1, \ldots, a_q \in Y$, and $\tau_i$ are integers satisfying $\sum_{i=1}^q|\tau_i|\le k$ for all $i=1, \ldots, q$. Let $\e \colon F(X)\to \ZZ=\langle t\rangle $ be the natural homomorphism with the kernel $\ll Y\rr$. Clearly, we have $\e(w^\prime)=\tau_0+\ldots +\tau_1 =0$. Therefore,
\begin{equation} \label{w''}
w^{\prime}=a_{1}^{t^{\sigma_1}}\ldots a_{q}^{t^{\sigma_q}}
\end{equation}
in the free group $F(X)$, where
$$
\sigma_1= -\tau_0=\tau_1+\tau_2+\ldots + \tau_q, \;\;\; \sigma_2= \tau_2+\ldots + \tau_q,\;\;\; \ldots\;\;\; \sigma_q=\tau_q.
$$
In particular, $|\sigma_i-\sigma_j| \le k$ for all $i$, $j$, and hence, $a_i^{t^{\sigma_i}}$ and $a_j^{t^{\sigma_j}}$ commute in $G_{\| w\|}$. Therefore, we can rewrite $w^\prime $ as $w^\prime= b_{1}^{t^{\rho_1}}\ldots b_{r}^{t^{\rho_r}}$ in $G_{\| w\|}$, where all elements $b_i$ belong to $Y$ and all exponents $\rho_i$ are pairwise distinct. Since this product represents $1$ in $G$, we obtain that $b_i=1$ for all $i$, i.e., $w^\prime =1 $ in $G_{\| w\|}$.
\end{proof}

From its presentation, we see that $G_k$ is an HNN-extension over a finite group (namely, the vertex group $K_{2 k+1}$ and the edge group $K_{2 k}$); in particular, it is virtually free and hence hyperbolic. This already implies the well-known fact that $G$ is approximated by hyperbolic groups (see \cite{Osi02}). The following lemma provides a quantitative estimate.

\begin{lem}\label{Lem:14k}
For every $k\in \NN$, the Cayley graph $\Cay(G_{k}, X)$ is a $14k$-hyperbolic metric space.
\end{lem}
\begin{proof}
Let
$$
B_1=K_{-k} \oplus  \cdots \oplus K_{k-1} \; \;\; \text{and} \;\; \; B_2=K_{-k+1} \oplus \cdots \oplus K_{k}.
$$
Note that $G_k$ is an HNN-extension of $A_k= K_{-k} \oplus  \cdots \oplus K_{k}$ with the associated subgroups $B_1$ and $B_2$ corresponding to the natural ``shift automorphism" $B_1\to B_2$. Let $T$ be the corresponding Bass-Serre tree of $G_k$. Recall that the vertex set of $T$ is $\{ gA_k\mid g\in G\}$ and, for every $g\in G$, the vertices $gA_k$ and $gt^{\pm1}A_k$ are connected by an edge. Let $v=1A_k$. Then the action of $G$ on $T$ satisfies the assumptions of Proposition \ref{Prop:Tree}.

Let us estimate the radius of $\Stab_G(v)=A_k$ with respect to $X$. Every $a\in A_k$ decomposes as $a=a_{-k}^{t^{-k}} \cdots a_{k}^{t^{k}}$ for some $a_{-k}, \ldots, a_k\in Y$.  We have
$$
|a|_X= |t^ka_{-k}t^{-1}a_{-k+1} \cdots  t^{-1}a_{k}t^{k}|_X\le 6k+1.
$$
Thus, $\Rad_X(A_k)\le 6k+1\le 7k$. By Proposition \ref{Prop:Tree}, $\Cay(G_k, X)$ is $14k$-hyperbolic.
\end{proof}

We are now ready to prove the main result of this section.

\begin{thm}\label{Thm:WP}
For any finite group $K$, the wreath product $K\wr \ZZ$ is well-approximated by hyperbolic groups. In particular, $f_{K\wr \ZZ} \sim n/m$.
\end{thm}

\begin{proof}
The presentation of $G=K\wr \ZZ $ described above satisfies conditions (a) and (b) of Definition \ref{Def:wah} by Lemmas \ref{Lem:k} and \ref{Lem:14k}, respectively. Thus, the theorem follows from Proposition \ref{Prop:LimHyp}.
\end{proof}


\subsection{Geodesics and detours in Cayley graphs}


In this section, we establish a property of groups with linear Dehn spectrum that is reminiscent of the exponential divergence of geodesic bi-infinite lines in hyperbolic spaces. It provides a useful tool for proving the non-linearity of the Dehn spectrum and will play a significant role in our treatment of groups of finite exponent in Section \ref{FEQI}.

We begin by introducing the necessary terminology. Throughout this section, $G$ denotes a group generated by a finite set $X$.

\begin{defn}\label{det}
Let $p$ be a geodesic path in $\Cay(G,X)$. We say that a path $q$ is an $r$-\emph{detour} of a vertex $o\in p$ if $q_-=p_-$, $q_+=p_+$, and $q$ does not intersect the $r$-neighborhood of $o$ (see Fig. \ref{Fig:det}).
\end{defn}

\begin{figure}
  \begin{center}
\begingroup%
  \makeatletter%
  \providecommand\color[2][]{%
    \errmessage{(Inkscape) Color is used for the text in Inkscape, but the package 'color.sty' is not loaded}%
    \renewcommand\color[2][]{}%
  }%
  \providecommand\transparent[1]{%
    \errmessage{(Inkscape) Transparency is used (non-zero) for the text in Inkscape, but the package 'transparent.sty' is not loaded}%
    \renewcommand\transparent[1]{}%
  }%
  \providecommand\rotatebox[2]{#2}%
  \newcommand*\fsize{\dimexpr\f@size pt\relax}%
  \newcommand*\lineheight[1]{\fontsize{\fsize}{#1\fsize}\selectfont}%
  \ifx\svgwidth\undefined%
    \setlength{\unitlength}{273.48658633bp}%
    \ifx\svgscale\undefined%
      \relax%
    \else%
      \setlength{\unitlength}{\unitlength * \real{\svgscale}}%
    \fi%
  \else%
    \setlength{\unitlength}{\svgwidth}%
  \fi%
  \global\let\svgwidth\undefined%
  \global\let\svgscale\undefined%
  \makeatother%
  \begin{picture}(1,0.31496822)%
    \lineheight{1}%
    \setlength\tabcolsep{0pt}%
    \put(0,0){\includegraphics[width=\unitlength,page=1]{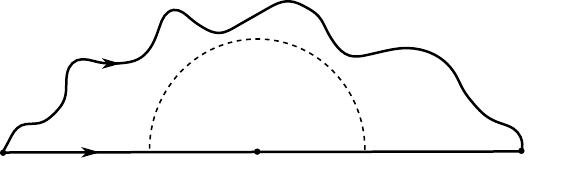}}%
    \put(0.18007895,0.2293365){\color[rgb]{0,0,0}\makebox(0,0)[lt]{\lineheight{1.25}\smash{\begin{tabular}[t]{l}$q$\end{tabular}}}}%
    \put(0,0){\includegraphics[width=\unitlength,page=2]{fig5.pdf}}%
    \put(0.49041473,0.11152403){\color[rgb]{0,0,0}\makebox(0,0)[lt]{\lineheight{1.25}\smash{\begin{tabular}[t]{l}$r$\end{tabular}}}}%
    \put(0.14018333,0.00588466){\color[rgb]{0,0,0}\makebox(0,0)[lt]{\lineheight{1.25}\smash{\begin{tabular}[t]{l}$p$\end{tabular}}}}%
    \put(-0.00210145,0.00881736){\color[rgb]{0,0,0}\makebox(0,0)[lt]{\lineheight{1.25}\smash{\begin{tabular}[t]{l}$p_-$\end{tabular}}}}%
    \put(0.90401723,0.0119898){\color[rgb]{0,0,0}\makebox(0,0)[lt]{\lineheight{1.25}\smash{\begin{tabular}[t]{l}$p_+$\end{tabular}}}}%
    \put(0.44212203,0.00694865){\color[rgb]{0,0,0}\makebox(0,0)[lt]{\lineheight{1.25}\smash{\begin{tabular}[t]{l}$o$\end{tabular}}}}%
  \end{picture}%
\endgroup%

  \end{center}
  \vspace{-3mm}
  \caption{An $r$-detour.}\label{Fig:det}
\end{figure}

\begin{prop}\label{Prop:div}
Let $G$ be a group with linear Dehn spectrum generated by a finite set $X$. There exists a constant $D>0$ such that the following holds. Let $p$ be a geodesic path in $\Cay(G,X)$ and let $q$ be an $r$-detour of a vertex of $p$ for some $r\ge 0$. If the loop $pq^{-1}$ is $k$-contractible (in a sence of the Definition \ref{Def:k-fil}) for some $k\in \NN$, then $\ell (q)\ge D(r/k)^2$.
\end{prop}

We will need two lemmas. The proof of the first one is a straightforward exercise, and we leave it to the reader.

\begin{lem}\label{Lem:Lin}
Let $f\in \mathcal F$. Suppose that $f\sim n/m$. Then there exists a constant $C$ such that for every $(k,m,n)\in \mathbb T$, we have $f(k,m,n)\le C\lceil n/m\rceil$ for every $m\ge Ck$, where $\mathbb T=\{ (k,m,n)\in \mathbb N\times \mathbb N\times\mathbb N\mid m\ge k \}$ as in the Definition \ref{Def:IPG}.
\end{lem}

The second lemma was proved in \cite[Lemma 5]{Ols}.

\begin{lem}[Olshanskii]\label{Lem:Ols}
Let $\Delta $ be a van Kampen diagram over a group presentation, where every face has perimeter at most $M$. Suppose that $\partial \Delta =p_1p_2\ldots p_n$, where each $p_i$ is a geodesic path in $\Delta^{(1)}$. If there exists $i$ and a vertex $o$ on $p_i$ such that $o$ does not belong to the closed $r$-neighborhood of $\bigcup_{j\ne i} p_j$, then the number of faces in $\Delta$ is at least $4r^2/M^3$.
\end{lem}

The main result of this section follows from Olshanskii's lemma.

\begin{proof}[Proof of Proposition \ref{Prop:div}]
Let $C$ be the constant provided by Lemma \ref{Lem:Lin} and let $p$, $q$ be as in the proposition. By the definition of $f_{G,X}$, there exists a van Kampen diagram $\Delta$ over $\langle X\mid \mathcal{S}_{Ck}\rangle $ with the boundary label $\Lab(pq^{-1})$ and at most $f_{G,X}(k, Ck, \ell(pq^{-1}))$ faces. By abuse of notation, we think of $p$ and $q$ as subpaths of $\partial \Delta$.

Clearly $p$ is a geodesic path in $\Delta^{(1)}$. We think of $\partial \Delta$ as a geodesic polygon whose sides are $p$ and individual edges of $q^{-1}$. Since $q$ is an $r$-detour of a vertex of $p$, Lemmas \ref{Lem:Lin} and \ref{Lem:Ols} imply that
$$
4r^2\le(Ck)^3 f_{G,X}(k, Ck, \ell(pq^{-1}))\le C^4k^3 \left\lceil \frac{\ell(pq^{-1})}{Ck}\right\rceil\le C^4k^3 \left\lceil \frac{2\ell(q)}{Ck}\right\rceil,
$$
which yields the desired inequality for an appropriate constant $D$.
\end{proof}

We illustrate Proposition \ref{Prop:div} by proving the following. 

\begin{cor}
Let $G$ be a finitely generated group with a linear Dehn spectrum. Then $G$ cannot be decomposed into the direct product of two infinite finitely generated groups.
\end{cor}

\begin{proof}
 For the sake of contradiction, assume that $G= A \times B$, where $A$ and $B$ are infinite finitely generated groups. Let $A = \langle Y \rangle$ and $B= \langle Z \rangle$, where $Y$ and $Z$ are finite, let $X=Y \sqcup Z$. Let $D$ be a constant from the statement of Proposition \ref{Prop:div} for $\Cay(G, X)$. We choose any $r > 96/D$ and consider a geodesic path $p$ in $\Cay(G,X)$ of length $2r$, where $\Lab(p)=a$, and the word $a$ contains only letters from  $Y$. Let $q$ be an $r$-detour of a middle vertex of $p$, such that $\Lab(q)=bab^{-1}$, where $b$ is a word of length $2r$ containing only letters from $Z$. Clearly, the loop $pq^{-1}$ in $\Cay(G,X)$ is $4$-contractible. Therefore, we have $\ell (q)\ge D(r/4)^2$ by Proposition \ref{Prop:div}. This means that $6r \ge D(r/4)^2$, which contradicts the inequality $r > 96/D$. 
\end{proof}

\begin{rem}
A similar idea can be used to show that a finitely generated group with linear Dehn spectrum cannot contain an undistorted central element of infinite order.
\end{rem}


\section{A descriptive approach to isoperimetric functions}\label{Sec:DA}



\subsection{$F_\sigma$ relations on Polish spaces}


We begin by briefly reviewing the necessary background from descriptive set theory.  A\emph{ $G_\delta$-subset} (respectively, an \emph{$F_\sigma$-subset}) of a topological space is a subset that can be represented as a countable intersection of open sets (respectively, a countable union of closed sets).  A subset of a topological space is \emph{meager} if it is a union of countably many nowhere dense sets. A \emph{comeager} set is a set whose complement is meager. Equivalently, a subset of a topological space is comeager if it is a countable intersection of sets with dense interior.

A topological space is a \emph{Polish space} if it is separable and completely metrizable. The Baire category theorem (see~\cite[Theorem~8.4]{Kec}) states that in a nonempty Polish space, every comeager subset is dense, and in particular is non-empty.

A topological space is \emph{perfect} if it has no isolated points. We will need a result of Mycielski stating that for every meager relation on a non-empty perfect Polish space, there exist $2^{\aleph_0}$ unrelated elements. For our purpose, it is more convenient to formulate this result in terms of comeager sets.

\begin{thm}[Mycielski, { \cite[Theorem 19.1]{Kec}}]\label{Thm:Myc}
Let $X$ be a non-empty perfect Polish space. For any comeager subset $R\subseteq X\times X$, there exists a subset $C\subseteq X$ of cardinality $2^{\aleph_0}$ such that $\{(x,y)\in C\times C \mid x\ne y\}\subseteq R$.
\end{thm}

As usual, by a \emph{quasi-order} we mean any reflexive and transitive (but not necessarily antisymmetric) relation. Every quasi-order $Q$ on a set $X$ induces an equivalence relation $E_Q$ by the rule
$$
xE_Qy \;\; \Longleftrightarrow  \;\; xQy \;\,{\rm and }\;\, yQx.
$$
A quasi-order $Q$ on $X$ is said to be \emph{$F_\sigma$} if it is an $F_\sigma$ subset of $X\times X$. A subset $A\subseteq X$ is a \emph{$Q$-antichain} if distinct elements of $A$ are $Q$-incomparable, i.e., we have neither $xQy$ nor $yQx$ for any distinct $x,y\in A$. We will need the following corollary of Mycielski's theorem.

\begin{cor}\label{Cor:2dec}
Let $Q$ be an $F_\sigma$ quasi-order on a Polish space $X$. Suppose that $X$ contains two distinct dense $E_Q$-equivalence classes. Then every comeager subset $S\subseteq X$ contains a $Q$-antichain of cardinality $2^{\aleph_0}$.
\end{cor}

\begin{proof}
We first note that $X$ is perfect and nonempty since it contains two distinct and disjoint dense subsets. By the Baire category theorem, every comeager subset of a Polish space contains a dense $G_\delta$ subset. Passing to such a subset if necessary, we can assume that $S$ is dense $G_\delta$. Being a $G_\delta $-subset, $S$ is a Polish space itself with respect to the induced topology (see \cite[Theorem 3.11]{Kec}). Further, $S$ has no isolated points since it is a dense subset of a perfect space.

Let $R_1=(X\times X)\smallsetminus Q$. We want to show that $R_1$ is comeager in $X\times X$. Since $Q$ is $F_\sigma$ in $X\times X$, $R_1$ is $G_\delta$. Let $[a]$, $[b]$ be two distinct dense $E_Q$-equivalence classes in $X$. Since $a$ and $b$ are not equivalent, we have $(a,b)\notin Q$ or $(b,a)\notin Q$. Assume $(a,b)\notin Q$ for definiteness. The density of $[a]$ and $[b]$ implies that every non-empty open subset of $X\times X$ contains a pair $(a^\prime, b^\prime)$, where $a^\prime \in [a]$ and $b^\prime\in [b]$. Clearly, $(a^\prime, b^\prime)\notin Q$. Thus, $R_1$ is comeager in $X\times X$.

Similarly, the set $R_2=(X\times X)\smallsetminus \{ (y,x)\mid (x,y)\in Q\}$ is comeager in $X\times X$. Since $S$ is comeager in $X$, we have $S =\bigcap_{n\in \NN} S_n$, where each $S_n\subseteq X$ has a dense interior. Note that $S \times S = \bigcap_{n\in \NN} S_n\times S_n$ and each $S_n \times S_n$ has a dense interior. Hence,  $S\times S$ is also a comeager subset of $X\times X$. It follows that the set $R=R_1\cap R_2\cap (S\times S)$ is comeager in $X\times X$ and, consequently, in $S\times S$. By Mycielski's theorem applied to $R\subseteq S$, there exists a subset $C\subseteq S$ of cardinality $2^{\aleph_0}$ such that $\{(x,y)\in C\times C \mid x\ne y\}\subseteq R$. The latter inclusion implies that $C$ is a $Q$-antichain.
\end{proof}


\subsection{The Dehn spectrum of marked groups}\label{Sec:aaifgg}


We begin by recalling the definition of the space of finitely generated marked groups given in \cite{Gri}. For each $n \in \NN$, let $\mathcal G_n$ denote the set of isomorphism classes of pairs $(G,X)$, where $G$ is a group, $X$ is a generating $n$-tuple of $G$, and two such pairs $(G, (x_1, \ldots, x_n))$ and $(H,(y_1, \ldots, y_n))$ are \emph{isomorphic} if there is an isomorphism $G \rightarrow H$ mapping $x_i$ to $y_i$ for every $i=1, \ldots, n$. Following the standard practice, we keep the notation $(G,X)$ for the equivalence class of $(G,X)$.

\begin{ex}
Let $n=2$ and let $G=\mathbb Z\oplus\mathbb Z$. For any generating tuples $X=(x_1,x_2)$ and $Y=(y_1,y_2)$ of $\mathbb Z\oplus\mathbb Z$, the pairs $(G, X)$ and $(G, Y)$ are isomorphic as the bijection between any two bases of a free abelian group extends to an isomorphism. On the other hand, there are infinitely many non-isomorphic pairs of the form $(G, T)$, where $G=\mathbb Z\oplus\mathbb Z$ and $|T|=3$. For instance, the pairs $(G, \{(1,0), (0,1), (1,n)\})$, $n\in \NN$, represent distinct elements of $\mathcal G_3$.
\end{ex}

The topology on $\G_n$ can be defined in terms of Cayley graphs as follows. Let $G$ be a group generated by an $n$-tuple $X$. Recall that $|g|_X$ denotes the word length of an element $g\in G$ with respect to $X$. For any $r>0$, we define $B_{G,X}(r)=\{ g\in G \mid |g|_X\le r\}.$

\begin{defn}\label{CayTop}
Let $(G,X), (H,Y)\in \G _n$, where $X=(x_1, \ldots, x_n)$ and $Y=(y_1, \ldots, y_n)$. We write $(G,X)\cong_r(H,Y)$ if there is an isomorphism between the induced subgraphs of $\Cay (G,X)$ and $\Cay (H,Y)$ with the vertex sets $B_{G,X}(r)$ and $B_{H,Y}(r)$, respectively, that sends edges labelled by $x_i$ to edges labelled by $y_i$ for all $i=1, \ldots, n$. The topology on $\G_n$ is defined by taking the family of sets $\{ W(G,X,r)\mid (G,X)\in \G_n,\; r\in \NN\}$, where
\begin{equation}\label{nbd}
W(G,X,r) = \{ (H,Y)\in \G_n\mid (G,X)\cong_r(H,Y)\},
\end{equation}
as the basis of neighborhoods.
\end{defn}

It is straightforward to verify that every $\G_n$ is Hausdorff, compact, and metrizable (see \cite{Gri} for details). In particular, $\G_n$ is a Polish space. We record an immediate yet useful observation.

\begin{lem}\label{Ex:quot}
Let $(G,X)\in \mathcal G_n$ and let $\e\colon G\to Q$ be an epimorphism. If the restriction of $\e$ to $B_{G,X}(r)$ is injective, then $(Q, \e(X))\in W(G,X,r)$.
\end{lem}

Recall that $\mathbb T=\{ (k,m,n)\in \mathbb N\times \mathbb N\times\mathbb N\mid m\ge k \}$, and $\mathcal F$ denotes the set of all functions $\mathbb T\to \mathbb R$ that are non-decreasing in the first and third arguments and non-increasing in the second argument. In what follows, we think of $\mathcal F$ as a topological space with the topology of pointwise convergence. For every $n\in \NN$, we consider the map $\iota_n\colon \G_n\to \mathcal F$ given by
$$
\iota_n(G,X)=f_{G,X}.
$$
This map is well-defined since the Cayley graphs of isomorphic pairs $(G,X)$ and $(H,Y)$ are isomorphic as unlabeled graphs.

\begin{lem}
For every $n\in \NN$, the map $\iota_n$ is continuous.
\end{lem}

\begin{proof}
It follows from the definition of a Dehn spectrum that, for any $r>0$ and any $(G_1,X_1), (G_2, X_2) \in \G_n$ such that $(G_1,X_1)\cong_r (G_2, X_2)$, we have $f_{G_1, X_1}(k,m,n)=f_{G_2, X_2}(k,m,n)$ for all $k,m,n\le r$. This implies the claim of the lemma.
\end{proof}

\begin{prop}\label{Prop:Fs}
For every $n\in \NN$, the following hold.
\begin{enumerate}
\item[(a)] The relation $\preccurlyeq_{\mathrm{Ds}}$ is an $F_\sigma$ quasi-order on $\G_n$.
\item[(b)] For every $g\in \mathcal F$, the set $\{ (G,X)\in \G_n\mid f_{G,X}\preccurlyeq g\}$ is an $F_\sigma$ subset of $\G_n$.
\end{enumerate}
\end{prop}

\begin{proof}
We begin with part (a). By definition, we have $(G,X)\preccurlyeq_{\mathrm{Ds}} (H,Y)$ for some $(G,X), (H,Y)\in \G_n$ if and only if $\iota_n(G,X)\preccurlyeq\iota_n(H,Y)$. Since $\iota_n$ is continuous, it suffices to show that $\preccurlyeq$ is an $F_\sigma$ quasi-order on $\mathcal F$. For two constants $C,D\in \NN$, we denote by $Q_{C,D}$ the set of all pairs $(f,g)\in \mathcal F\times \mathcal F$ such that
\begin{equation}\label{Eq:Fs1}
f(k,Cm,n)\le Cg(Ck, m, Cn) +Cn/m +C
\end{equation}
for all $(k,m,n)\in \mathbb T$ satisfying the inequalities
\begin{equation}\label{Eq:Fs2}
k\le D,\;\;\; Ck\le m\le D,\;\;\; n\le D.
\end{equation}
As a subset of $\mathcal F \times \mathcal F$, the relation $\preccurlyeq$ coincides with $\bigcup\limits_{C\in \NN}\bigcap\limits_{D\in \NN} Q_{C,D}$. Since there are only finitely many $(k,m,n)\in \mathbb T$ satisfying (\ref{Eq:Fs2}), each $Q_{C,D}$ is closed in $\mathcal F\times \mathcal F$ and we obtain (a).

The proof of part (b) is similar. For any two constants $C,D\in \NN$ and the fixed function $g\in \mathcal F$, we denote by $R_{C,D}$ the set of all $f\in \mathcal F$ satisfying (\ref{Eq:Fs1}) for all $(k,m,n)\in \mathbb T$ satisfying (\ref{Eq:Fs2}). In this notation, we have
$\{ (G,X)\in \G_n\mid f_{G,X}\preccurlyeq g\}=  \bigcup\limits_{C\in \NN}\bigcap\limits_{D\in \NN} R_{C,D}$
and the claim follows as in part (a).
\end{proof}

Recall that
$$
{\mathcal H}_n=\{ (G,X)\in \mathcal G_n\mid G \textrm{ is non-elementary hyperbolic} \}
$$
and $\overline{\mathcal H}_n$ denotes the closure of $\mathcal H_n$ in $\mathcal G_n$.
We want to apply Corollary \ref{Cor:2dec} to the closure $\overline{\mathcal H}_n$ of ${\mathcal H}_n$ in $\G_n$ and the quasi-order $\preccurlyeq_{\mathrm{Ds}}$. To this end, we need Lemma \ref{Lem:quad} (see below) proved by employing the Dehn filling technology. We briefly discuss the necessary background.

A group $G$ is \emph{hyperbolic relative to a subgroup} $H$ if it is finitely presented relative to $H$ and satisfies a linear relative isoperimetric inequality. That is, there exist finite subsets $X\subseteq G$, $\mathcal R\subseteq F(X)\ast H$, and a constant $C\in \NN$ such that the natural homomorphism $F(X) * H \rightarrow G$ induces an isomorphism $(F(X) * H) /\langle\langle\mathcal{R}\rangle\rangle \rightarrow G$, and every element $w\in \ll \mathcal R\rr$ can be represented in $F(X)\ast H$ as the product of at most $C|w|_{X\cup H}$ conjugates of elements of $\mathcal R$ or their inverses. It is well-known that every group hyperbolic relative to a hyperbolic subgroup is hyperbolic itself (see, for example, \cite[Corollary 2.41]{Osi06}).

We will need the following result. Related statements appear in the literature (see, for example, \cite[Corollary 3.26]{CIOS}), but the precise formulation required here does not seem to be available, so we include a brief proof for completeness. The notion of a hyperbolically embedded subgroup used in the proof can be treated as a ``black box"; a detailed discussion of this notion may be found in the introduction to \cite{DGO}.

\begin{lem}\label{Lem:Fhe}
Let $G$ be a non-elementary hyperbolic group. For every $k\in \NN$, $G$ contains a subgroup $H$ such that $H$ is isomorphic to a direct product of a finite group $K$ and a free group of rank $k$ and $G$ is hyperbolic relative to $H$.
\end{lem}

\begin{proof}
By \cite[Theorem 2.24]{DGO}, $G$ contains a hyperbolically embedded subgroup $H$ with the prescribed structure. By \cite[Theorem 4.31]{DGO}, $H$ is finitely generated and undistorted in $G$; the latter condition means that there are finite generating sets $A$, $B$ of $G$ and $H$, respectively, and a constant $C$ such that $|h|_B\le C|h|_A$ for all $h\in H$. Further, we have $|H\cap g^{-1}Hg|<\infty $ for all $g\in G\setminus H$ by \cite[Proposition 2.10]{DGO}. These conditions allow us to apply {\cite[Theorem 1.5]{Osi06}} and to conclude that $G$ is hyperbolic relative to  $H$.
\end{proof}

The following result summarizes \cite[Theorem 1.1 and Corollary 1.2]{Osi07}.

\begin{thm}[Osin, \cite{Osi07}]\label{Thm:DF}
Let $G$ be a group hyperbolic relative to a subgroup $H$ and let $\S$ be a finite subset of $G$. There exists a finite set $ \mathcal T \subseteq H\smallsetminus\{ 1\}$ such that for any subgroup $N\lhd H$ satisfying $N\cap \mathcal T=\emptyset$, the following hold.
\begin{enumerate}
\item[(a)] Let $\ll N\rr$ denote the normal closure of $N$ in $G$. Then $\ll N\rr \cap H =N$, and the quotient group $G/\ll N\rr$ is hyperbolic relative to the image of $H$. In particular, if the quotient group $H/N$ is hyperbolic, then so is $G/\ll N\rr$.
\item[(b)] The natural homomorphism $G\to G/\ll N\rr$ is injective on $\S$.
\end{enumerate}
\end{thm}

Let $\mathcal C$ be a class of groups. A group $G$ is said to be \emph{fully residually $\mathcal C$} if, for every finite $\S\subseteq G$, there exists an epimorphism $\e\colon G\to Q$ such that $Q\in \mathcal C$ and the restriction of $\e$ to $\S$ is injective.

We denote by $\mathcal{NEH}$ the set of all non-elementary hyperbolic groups. Note that if $(G,X)\in \G_n$ and $G$ is fully residually $\mathcal{NEH}$, then $(G,X)\in \overline{\mathcal H}_n$, where $n=|X|$. The following result is an immediate corollary of Theorem \ref{Thm:DF}.

\begin{cor}\label{Cor:FRH}
Let $G$ be a group generated by a finite set $X$. Suppose that $G$ is hyperbolic relative to a fully residually $\mathcal{NEH}$ group $H$. Then $G$ is fully residually $\mathcal{NEH}$. In particular, $(G,X)\in \overline{\mathcal H}_n$, where $n=|X|$.
\end{cor}

\begin{lem}\label{Lem:quad}
For every integer $n\ge 2$, the set $\overline{\mathcal H}_n$ contains a dense subset of marked groups with quadratic Dehn function.
\end{lem}

\begin{proof}
Let $U$ be a non-empty open subset in $\overline{\mathcal H}_n$, $V$ an open subset in $\G_n$ such that $U=V\cap \overline{\mathcal H}_n$. Let $(G,X) \in U$ be a marked hyperbolic group. By the definition of the topology in $\G_n$, there exists $r\in \NN$ such that $W(G,X,r)\subseteq V$, where $W(G,X,r)$ is defined by (\ref{nbd}).

For $k\in \NN$, let $F_k $ denote the free group of rank $k$. By Lemma \ref{Lem:Fhe}, for every $k\in \NN$, there exists a subgroup $H\cong F_k\times K$ of $G$ such that $|K|<\infty$ and $G$ is hyperbolic relative to $H$. We apply this result to $k=14$ and denote by $A=\{a_1, \ldots , a_7, b_1, \ldots , b_7\}$ a basis of $F_{k}$. Henceforth, we identify $H$ with $F_k\times K$ and think of $F_k$ as a subgroup of $H$.

Let $\S=B_{G,X}(r)$ and let $\mathcal T\subseteq H\smallsetminus\{ 1\}$ be the set provided by Theorem \ref{Thm:DF}. We choose $m\in \NN$ such that all elements of $\mathcal T\cap F_k$ have length at most $m$ with respect to $A$. Let
$$
u=a_1a_2^m\ldots a_7^m, \;\;\;\; v=b_1b_2^m\ldots b_7^m.
$$
Note that $\langle u,v\rangle $ is a free factor of $F_k$ and the presentation $\langle A\mid u,\, v\rangle $ satisfies the $C^\prime(1/6)$ small cancellation condition. Let $M$ (respectively, $N$) denote the normal closure of the set $\{ u, v\}$ (respectively, the normal closure of the commutator $[u, v]$) in $F_k$. By the Greendlinger lemma \ref{Lem:Green} and the choice of $m$, we have $N\cap \mathcal T\subseteq M\cap \mathcal T=\emptyset$. Note that $N$ is also normal in $H$. Applying Theorem \ref{Thm:DF} (a) and taking into account that $\langle u,v\rangle $ is a free factor of $F_k$, we conclude that the quotient group $Q=G/\ll N\rr$ is hyperbolic relative to a subgroup
$$
R\cong H/N\cong F_k/N\times K \cong ((\ZZ \times \ZZ)\ast F_{12})\times K.
$$
Furthermore, part (b) of Theorem \ref{Thm:DF} implies that the restriction of the natural homomorphism $\e\colon G\to Q$ to $\S$ is injective. Therefore, we have $(Q, \e(X))\in W(G,X,r)\subseteq V$ (see Lemma \ref{Ex:quot} (a)).

Note that the Dehn function of $R$ is quadratic. Since $Q$ is hyperbolic relative to $R$, the Dehn function of $Q$ is at most quadratic by \cite[Corollary 2.41]{Osi06}. In fact, it is exactly quadratic; indeed, $Q$ cannot be hyperbolic since hyperbolic groups do not contain subgroups isomorphic to $\ZZ\times \ZZ$. Finally, taking the quotients of $R$ by the normal closures of finite index subgroups of $\ZZ \times \ZZ\le R$, one can show that $R$ is fully residually non-elementary hyperbolic. By Corollary \ref{Cor:FRH}, we obtain $(Q,\e(X))\in \overline{\mathcal H}_n$. Thus, we have $(Q,\e(X))\in U$.
\end{proof}

\begin{rem}
Elaborating on the proof of Lemma \ref{Lem:quad}, one can show that $\overline{\mathcal H}_n$ contains a dense subset consisting of groups each of which is hyperbolic relative to a subgroup isomorphic to $\mathbb Z^2\times K$ for some finite group $K$. Indeed, this follows from the fact that $((\ZZ \times \ZZ)\ast F_{12})\times K$ is hyperbolic relative to $\mathbb Z^2\times K$, which can be established using \cite[Theorem 2.40]{Osi06}.
\end{rem}

\begin{lem}[Osin, \cite{Osi21}]\label{Osi21}
For every integer $n\ge 2$, a generic element $(G,X)\in \overline{\mathcal H}_n$ is approximated by hyperbolic groups in the sense of Definition \ref{Def:wah}.
\end{lem}

\begin{proof}
This result was obtained in the proof of  \cite[Theorem 2.11]{Osi21} for a topological union $\overline{\mathcal H}=\bigcup_{n\in \NN}\overline{\mathcal H}_n$ in place of $\overline{\mathcal H}_n$, where $\overline{\mathcal{H}}_n$ appears with the same topology as here, as an open subset of $\overline{\mathcal{H}}$. The claim thus follows.
\end{proof}

\begin{proof}[Proof of Theorem \ref{Thm:HL}]
Let $\sim_{\mathrm{Ds}}$ denote the equivalence relation induced by the quasi-order $\preccurlyeq_{\mathrm{Ds}}$. Combining Corollary \ref{Cor:Hyp}, Corollary \ref{Cor:Quad}, and Lemma \ref{Lem:quad}, we obtain that $\overline{\mathcal H}_{n}$ contains two distinct dense $\sim_{\mathrm{Ds}}$-equivalence classes: one consisting of groups with linear Dehn spectrum and the other consisting of groups with Dehn spectrum equivalent to $(n/m)^2$. Together with Proposition \ref{Prop:Fs}, this allows us to apply Corollary \ref{Cor:2dec} to the quasi-order $\preccurlyeq_{\mathrm{Ds}}$ on $\overline{\mathcal H}_{n}$, which yields part (a) of the theorem.

By Lemma \ref{Osi21}, there is a comeager subset $C\subseteq \overline{\mathcal H}_n$ such that, for every $(G,X)\in C$, the group $G$ is approximated by hyperbolic groups. Let $$C_0= \{ (G,X)\in C\mid f_{G,X}(k,m,n)\sim n/m\}= \{ (G,X)\in C\mid f_{G,X}(k,m,n)\preccurlyeq n/m\}.$$ Corollary \ref{Cor:LHk} guarantees that $f_{G,X}(k,m,n)$ essentially depends on $k$ for every $(G,X)\in C\smallsetminus C_0$. By Proposition \ref{Prop:Fs} (b), $C_0$ is an $F_\sigma$ subset of $\mathcal G_n$. Further, by Lemma \ref{Lem:quad}, $C_0$ is nowhere dense in $\mathcal G_n$. Thus, $C_0$ is meager, $C\smallsetminus C_0$ is comeager, and we obtain the second part of the theorem.
\end{proof}

\begin{rem}\label{Rem:appr}
Part (b) of Theorem \ref{Thm:HL} together with Lemma \ref{Osi21} immediately imply that there exists a finitely generated group $G$ approximated by hyperbolic groups such that $f_G(k,m,n)\not\sim n/m$. This contrasts with our result about groups well-approximated by hyperbolic ones (see Proposition \ref{Prop:LimHyp} (b)).
\end{rem}

Using a similar approach, we obtain the following proposition, which is inspired by \cite[Proposition 5.2]{MOW} (for a similar yet different idea see also \cite[Proposition 4.4]{BC}) and will be used in the proof of Theorem \ref{Thm:FE}. We believe it is also of independent interest and may have other applications.  

\begin{prop}\label{Prop:Cext}
Let $G$ be a finitely generated group such that the center of $G$ decomposes as $$Z(G)=\bigoplus\limits_{i\in \mathbb N} A_i,$$
where each $A_i$ is a non-trivial finite abelian group. For each subset $I\subseteq \mathbb N$, we define $$Z_I =\left\langle \bigcup\limits_{i\in I}A_i \right\rangle\;\;\;\; {\rm  and }\;\;\;\; G_I=G/Z_I.$$
If $f_G\not\sim f_{G/Z(G)}$, then the set $\{ G_I\mid I\subseteq \NN\}$ contains $2^{\aleph_0}$ groups with pairwise incomparable Dehn spectra.
\end{prop}

\begin{proof}
We fix any finite generating set $X$ of $G$. For every $I\subseteq \NN$,  let $n=|X|$ and let $X_I$ denote the image of $X$ in $G_I$. We also denote by $\mathcal C$ the set $2^\NN$ endowed with the product topology. It can be shown that the map $\mathcal{C}\to \mathcal G_n$ given by $I\mapsto (G_I,X_I)$ is continuous (see \cite[Lemma 5.1]{MOW}). In particular, the set $S=\{ G_I\mid I\subseteq \NN\}$ is closed in $\mathcal G_n$ being a continuous image of a compact set, and thus is a Polish space.

The sets $$C_1=\{I \mid |I|<\infty\}\;\;\;\; {\rm and}\;\;\;\; C_2=\{I \mid |\NN\smallsetminus I|<\infty\}$$ are both dense in $\mathcal C$. Therefore, the sets $$S_0=\{ (G_I, X_I)\mid |I|<\infty\}\;\;\;\; {\rm and}\;\;\;\;  S_1=\{ (G_I, X_I)\mid |\NN\smallsetminus I|<\infty\}$$ are dense in $S$. If $I\in C_1$, the group $G_I$ is the quotient of $G$ by the finite normal subgroup $\bigoplus_{i\in I} A_i$ and, therefore, is quasi-isometric to $G$. Similarly, if $I\in C_2$, then the group $G_I$ is quasi-isometric to $G/Z(G)$ being an extension of a finite group by $G/Z(G)$. By Theorem~\ref{Thm:QI}, each of the subsets $S_1$ and $S_2$ consists of $\sim_{\mathrm{Ds}}$-equivalent groups. Combining this with the assumption $f_G\not\sim f_{G/Z(G)}$, we obtain that $S$ contains two distinct dense $\sim_{\mathrm{Ds}}$-classes. It remains to apply Corollary \ref{Cor:2dec} to the quasi-order $\preccurlyeq_{\mathrm{Ds}}$ on $S$.
\end{proof}


\section{Groups of finite exponent}\label{Sec:GFE}



\subsection{Olshanskii's graded small cancellation theory}\label{Sec:O}


Our next goal is to prove the linearity of the Dehn spectrum of free Burnside groups of sufficiently large odd exponent. This result will then be used to prove Theorem \ref{Thm:FE}. Our proofs make use of Olshanskii's graded small cancellation technique used in the geometric proof of the Novikov-Adian theorem \cite{Ols83}. To make our exposition as self-contained as possible, we review the necessary background from \cite{book}.

The \emph{free Burnside group} $B(r,N)$ of exponent $N$ and rank $r$ is the free group in the variety of groups of exponent $N$. Throughout the rest of this section, we assume that $N$ is odd and large enough, and $r\ge 2$. By \cite[Theorem 19.1]{book}, the group $B(r,N)$ can be defined by the presentation
\begin{equation}\label{B}
B(r,N)=\left\langle X  \; \left| \; R=1,\, R\in \mathcal R=\bigcup\limits_{i=1}^\infty
\mathcal R_i\right.\right\rangle ,
\end{equation}
where $X=\{ x_1, \ldots, x_r\}$ and the sets of relators $\mathcal R_0\subseteq \mathcal R_1\subseteq \ldots$  are constructed as follows.
Let $\mathcal R_0 =\emptyset $. By induction, suppose that $i> 0$ and we
have already defined the sets $\mathcal R_{j}$ for all  $j=0,\ldots,i-1$. Let $G_{i-1}=\langle X \mid R=1, R\in \mathcal R_{i-1}\rangle $. We define $$\mathcal R _{i}=\mathcal R _{i-1} \cup \{ P^N\mid P\in \mathcal P_i\},$$ where the set of \emph{periods} $\mathcal P_{i}$ is a maximal set of words in the alphabet $X\cup X^{-1}$ satisfying the following conditions.

\begin{enumerate}
\item[(a)] $\mathcal P_{i}$ consists of words of length $i$ that are not conjugate to powers of words of length less than $i$ in $G_{i-1}$.

\item[(b)] If $A,B\in \mathcal P_{i}$ and $A\not\equiv B$, then $A$ is not conjugate to $B^{\pm 1}$ in the group $G_{i-1}$.
\end{enumerate}

Results of Chapters 5 and 6 of \cite{book} involve a sequence of fixed small parameters
\begin{equation}\label{para}
1/2 > \alpha >\beta >\gamma> \e> \zeta >\iota=1/N.
\end{equation}
The exact relations between these parameters are described by a system of inequalities, which can be made consistent by choosing each parameter in the sequence to be sufficiently small compared to all preceding parameters. We can also assume that the parameters satisfy any additional finite system of inequalities that can be made consistent in this way; this is called the \emph{lowest parameter principle} in \cite[Section 15.1]{book}. To help the reader follow our proofs, we use the abbreviation \emph{LPP} to designate inequalities that involve the lowest parameter principle.

Let
\begin{equation}\label{Gi}
G_i=\langle X\mid \mathcal R_i\rangle .
\end{equation}
Van Kampen diagrams over presentations (\ref{B}) and (\ref{Gi}) considered in \cite{book} can also have edges labeled by $1$ and the so-called \emph{$0$-faces} whose labels represent $1$ in the free group $F(X)$. Edges labelled by $1$ are assumed to be of length $0$ and are not counted when computing lengths of paths and distances in van Kampen diagrams. Faces labelled by elements of $\mathcal R$ are called \emph{ $\mathcal R$-faces}.

Let $i\ge 1$ and suppose we have a pair of distinct faces $\Pi_1$ and $\Pi_2$ in a diagram $\Delta$ over (\ref{B}) such that their boundaries $p_1$ and $p_2$ (oriented in the same way, e.g., clockwise) are labeled by $P^N$ and $P^{-N}$ for some period $P\in \mathcal P_i$. Further, suppose that there is a simple path $t$ in $\Delta $ from $(p_1)_-$ to $(p_2)_-$ such that $T\equiv \Lab(t)$ represents $1$ in $G_{i-1}$. Then $$\Lab(p_1tp_2t^{-1})\equiv P^NTP^{-N}T^{-1}=1$$ in $G_{i-1}$ and, therefore, the subdiagram of $\Delta $ bounded by $p_1tp_2t^{-1}$ can be replaced with a diagram $\Sigma$ over $G_{i-1}$ (see Fig. \ref{figred}).

We call such $\mathcal R$-faces $\Pi_1$ and $\Pi_2$  a \emph{reducible pair} and refer to the procedure described in the previous paragraph as a \emph{reduction}. A Kampen diagram is \emph{reduced} if it contains no reducible pairs of faces. Note that every van Kampen diagram can be transformed into a reduced one by finitely many reductions. Indeed, every reduction lexicographically decreases the {\it type} $\tau(\Delta)=(\tau_1,\tau_2,\dots)$ of the diagram, where $\tau_j$ is the number of faces in $\Delta$ labelled by relators from $\mathcal R_j$.

\begin{figure}
  \begin{center}
  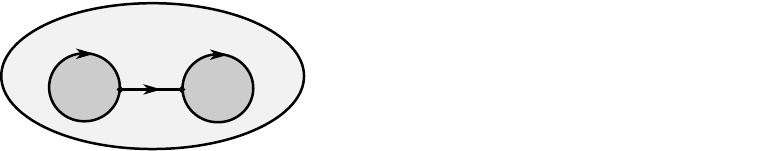
  \end{center}
  \vspace*{-3mm}
  \caption{A reduction}\label{figred}
 \end{figure}

Next, we discuss the notion of a contiguity subdiagram. The precise definition is somewhat technical and is not used in our paper. For our purposes, it suffices to know that a \emph{contiguity subdiagram} $\Gamma$ of an $\mathcal R$-face $\Pi$ to a subpath $q$ of $\partial\Delta$ is a subdiagram of $\Delta $ bounded by a path $s_1p_1s_2p_2$, where $p_1$ and $p_2$ are subpaths of $\partial\Pi_1$ and $q$, respectively, and $\Gamma$ does not contain $\Pi$ (see Fig. \ref{figcs}). In what follows, we call $$\partial \Gamma =s_1p_1s_2p_2$$ the \emph{canonical boundary decomposition} of $\Gamma$. The paths $s_1$ and $s_2$ (respectively, $p_1$ and $p_2$) are called the \emph{side sections} (respectively, \emph{contiguity sections}) of $\partial \Gamma$). The ratio $\ell(p_1)/\ell(\partial \Pi )$ is called the {\it contiguity degree} of $\Pi $ to $\partial \Delta$ and is denoted by $(\Pi , \Gamma , \partial \Delta)$.

By \cite[Lemma 19.4]{book}, every reduced diagram $\Delta$ over the presentation (\ref{B}) is an \emph{$A$-map}. Again, we do not discuss the definition of an $A$-map since we do not use it here. We only need to know that all lemmas about $A$-maps stated in Chapter 5 of \cite{book} apply to reduced diagrams. Below, we collect some necessary results.

\begin{figure}
  \begin{center}
  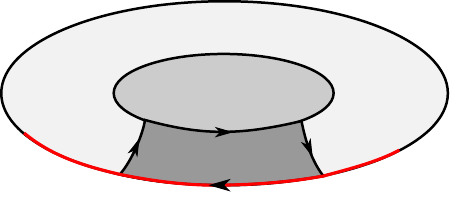
  \end{center}
  \vspace*{-3mm}
  \caption{A contiguity subdiagram}\label{figcs}
 \end{figure}

\begin{lem}[{\cite[Corollary 17.1]{book}}]\label{beta} If a reduced diagram $\Delta$ over (\ref{B}) contains an $\mathcal R$-face $\Pi$, then $\ell(\partial\Delta)>(1-\beta)\ell(\partial\Pi)$.
\end{lem}

\begin{lem}\label{cont}
Let $\Delta$ be a reduced diagram over (\ref{B}) and let $\Gamma$ be a contiguity subdiagram of an $\mathcal R$-face $\Pi$ to a subpath $p$ of $\partial\Delta$ with the canonical boundary decomposition  $\partial(\Pi,\Gamma,q)=s_1p_1s_2p_2$. Then the following holds.
\begin{enumerate}
\item[(a)] \cite[Lemma 15.3]{book}. $\max (\ell(s_1),\ell(s_2))<\zeta\ell(\partial\Pi)$.

\item[(b)] \cite[Lemma 15.4]{book}. If $(\Pi, \Gamma, q)\ge \varepsilon$, then $\ell(p_1)<(1+2\beta)\ell(p_2)$.

\item[(c)] \cite[Lemma 15.6]{book}. If $(\Pi, \Gamma, q)\ge 1/2+\alpha$ and $\partial \Pi= p_1t$, then $\ell(s_1t^{-1}s_2)< \ell(p_2)$.
\end{enumerate}
\end{lem}

We say that a (combinatorial) path $p$ in a van Kampen diagram $\Delta$ is \emph{geodesic} if it is geodesic in $\Delta^{(1)}$. We will need an immediate consequence of Lemma \ref{cont} (c) and the inequality $1/2+\alpha < 1- \gamma$ (LPP).

\begin{cor}\label{geod}
Let $\Delta$ be a van Kampen diagram over (\ref{B}), $q$ a subpath of $\partial \Delta$, $\Gamma$ a contiguity subdiagram of an $\mathcal R$-face $\Pi$ to $q$. If $q$ is a geodesic path in $\Delta$, then $(\Pi, \Gamma, q)<1/2+\alpha < 1-\gamma $.
\end{cor}

The next lemma is an immediate consequence of \cite[Theorem 16.1]{book} (for $\ell \le 4$ it is stated as Corollary 16.1 in \cite{book}).  Note that the proof of Theorem 16.1 given in \cite{book} works for every $\ell\in \NN$, but we may have to decrease the parameters (\ref{para}) and use the lowest parameter principle explained above to make sure that all the inequalities involved in the proof are consistent as $\ell $ increases. For this reason, we need $\ell $ to be fixed in advance. We will need the case $\ell\le 6$ in our paper.

\begin{lem}\label{gamma}
Let $\Delta$ be a reduced diagram over (\ref{B}) containing at least one $\mathcal R$-face. Suppose that the boundary of $\Delta $ decomposes as $\partial \Delta =q_1\ldots q_\ell$, where $\ell\le 6$ . Then there is an $\cal R$-face $\Pi$ of $\Delta$ and disjoint contiguity subdiagrams $\Gamma_1,\ldots, \Gamma_\ell$ of $\Pi$ to $q_1,\ldots,q_\ell$, respectively (some of them may be absent) such that
$$\sum_{i=1}^\ell(\Pi,\Gamma_i,q_i)>1-\gamma.$$
\end{lem}

\begin{rem}
  Note that the exponent $N$ depends on parameters from (\ref{para}), which are chosen according to the lowest parameter principle. In \cite{book} odd $N > 10^{10}$ is sufficient. In our case, the lower bound on $N$ might be larger, since we consider the case $\ell\le 6$ in this paper.   
\end{rem}

 If sections $q_1, \ldots, q_\ell$ in Lemma \ref{gamma} are geodesic, Corollary \ref{geod} implies that at least two contiguity subdiagrams $\Gamma_i$ must exist. The lemma below will be used several times to deal with the case when exactly two contiguity subdiagrams are present. In what follows, we denote by $\d_\Delta $ the length function on $\Delta^{(1)}$.

\begin{lem}\label{corner}
Let $\Delta$ be a reduced diagram over (\ref{B}), $p$ and $q$ be geodesic subpaths of $\partial \Delta$. Suppose that $\Delta$ contains an $\cal R$-face $\Pi$ and disjoint contiguity subdiagrams $\Gamma_1$, $\Gamma_2$ of $\Pi$ to $p$ and $q$, respectively, such that
\begin{equation}\label{Eq:G1G2}
(\Pi,\Gamma_1,p)+(\Pi,\Gamma_2,q)>1-\gamma.
\end{equation}
Then $\d_\Delta (p_-, q) < \ell(p)$.
\end{lem}

\begin{figure}
  \begin{center}
  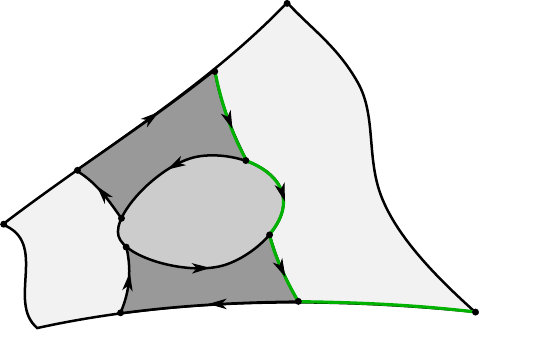
  \end{center}
  \vspace*{-3mm}
  \caption{The length of the green path is less than $\ell(p)$.}\label{figcor}
 \end{figure}

\begin{proof}
Let $\partial \Gamma_1=s_1p_1s_2p_2$ and $\partial \Gamma_2=t_1q_1t_2q_2$ be the canonical decompositions as shown on Fig. \ref{figcor}. By Corollary \ref{geod}, we have $(\Pi,\Gamma_2,q) <1/2+\alpha$. Hence, $(\Pi,\Gamma_1,p) > 1/2 - \gamma - \alpha  > \epsilon $ (LPP). Using Lemma \ref{cont} (b), we obtain
\begin{equation}\label{Eq:lq2}
\ell(p_2)> \frac{\ell(p_1)}{1+2\beta}> \frac{1/2 -\gamma-\alpha}{1+2\beta} \ell(\partial \Pi)> \frac 13 \ell(\partial \Pi) \;\;\; (LPP).
\end{equation}
Note that $\d_\Delta (p_-, q)$ is bounded from above by the length of the path labeled in green on Fig. \ref{figcor}. Let $z$ be the subpath of $\partial \Pi$ connecting $(t_1)_+$ to $(s_2)_-$ that does not contain $p_1$ and $q_1$. By (\ref{Eq:G1G2}), we have $$\ell(z) \le  \ell(\partial\Pi)-\ell(p_1)-\ell(q_1)< \gamma\ell(\partial\Pi).$$ Combining this with (\ref{Eq:lq2}) and Lemma \ref{cont} (a), we obtain
$$
\d_\Delta (p_-, q) \le \ell (t_1)+\ell(z) +\ell(s_2) + \ell(p) - \ell(p_2)< \ell(p) + (2\zeta +\gamma - 1/3) \ell(\partial \Pi) <\ell (p) \;\;\; (LPP).
$$
\end{proof}


\subsection{Dehn spectrum of free Burnside groups}


In this section, we compute the Dehn spectrum of free Burnside groups of sufficiently large odd exponent. A major step towards this goal is the following.

\begin{prop}\label{33ni}
Let $G_i=\langle X\mid \mathcal R_i\rangle $, where relations $R_i$ are defined in (\ref{B}). The Cayley graph $\Cay(G_i,X)$ is $7Ni$-hyperbolic.
\end{prop}

\begin{proof}
Consider an arbitrary geodesic triangle with vertices $x$, $y$, $z$ in $\Cay(G_i, X)$. We use the notation $[x,y]$, $[y,z]$, $[z,x]$ for the sides of this triangle. For any vertex $o\in [z,x]$, we will prove the inequality
\begin{equation}\label{doxyz}
\d_i(o, [x,y]\cup [y,z])\le 6.5 Ni,
\end{equation} where $\d_i$ denotes the metric on $\Cay(G_i, X)$. Without loss of generality, we can assume that
\begin{equation}\label{Eq:do}
\d_i(o, [x,y]\cup [y,z]) > 3Ni.
\end{equation}

\smallskip

\noindent{\it Case A.\;} Suppose first that the segment $[z,o]$ of the side $[z,x]$ intersects the closed $3Ni$-neighborhood of $[x,y]$, or the segment $[o,x]$ of $[z,x]$ intersects the closed $3Ni$-neighborhood of $[y,z]$. For definiteness, we consider the former case; the latter one is analogous.

Let $r_1$ be the shortest subpath of $[z,x]$ containing $o$ such that $$\d_i((r_1)_-, [x,y])= \d_i((r_1)_+, [x,y])=3Ni.$$ Let $u,v\in [x,y]$ be the vertices nearest to $(r_1)_+$ and $(r_1)_-$, respectively (see Fig. \ref{fig4-gon}). Let $r_2$ denote the subpath of $[x,y]^{\pm 1}$ connecting $u$ to $v$ and let $b_1$, $b_2$ be any geodesic paths in $\Cay(G_i,X)$ connecting $(r_1)_+$ to $u$ and $v$ to $(r_1)_-$. In what follows, we call the paths $r_1$, $r_2$ (respectively, $b_1$, $b_2$) \emph{red} (respectively, \emph{blue}). The quadrilateral $r_1b_1r_2b_2$ satisfies the following condition.

\begin{enumerate}
\item[($\ast$)] The distance in $\Cay(G_i, X)$ between distinct red sides and the length of every blue side are equal to $3Ni$.
\end{enumerate}

Let $\Delta $ be a reduced van Kampen diagram over (\ref{Gi}) with the boundary label $\Lab(\partial \Delta)\equiv \Lab(r_1b_1r_2b_2)$. By abuse of notation, we identify $\partial \Delta$ with its canonical image $r_1b_1r_2b_2$ in $\Cay(G_i,X)$. Note that $\Delta $ must contain at least one $\mathcal R$-face; for otherwise, $o$ belongs to $b_1r_2b_2$, which contradicts (\ref{Eq:do}). By Lemma \ref{gamma}, we can find an $\mathcal R$-face $\Pi $ and contiguity subdiagrams $\Gamma_1$, $\Gamma_2$, $\Sigma_1$, $\Sigma_2$ of $\Pi$ to $r_1$, $r_2$, $b_1$, $b_2$ respectively (some of these diagrams may be absent) such that $$\sum\limits_{j=1}^2\Big((\Pi,\Gamma_j, r_j) + (\Pi, \Sigma_j, b_j)\Big) \ge 1-\gamma$$ (see Fig. \ref{fig4-gon}.)

\begin{figure}
  \begin{center}
  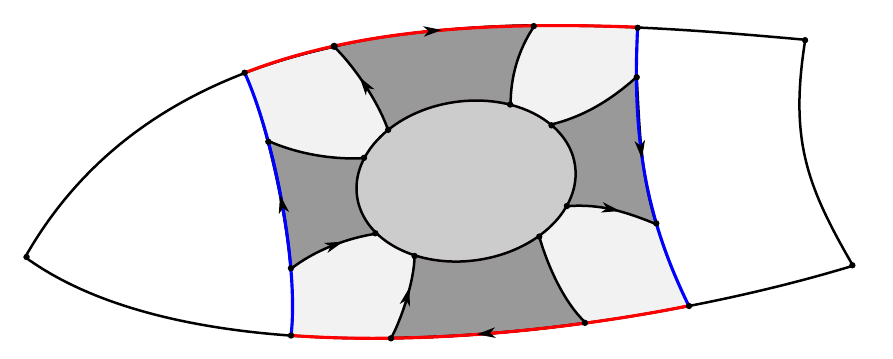
  \end{center}
  \vspace*{-3mm}
  \caption{Case A.}\label{fig4-gon}
 \end{figure}

By Corollary \ref{geod}, at least two of the contiguity subdiagrams $\Gamma_1$, $\Gamma_2$, $\Sigma_1$, $\Sigma_2$ must actually exist. Furthermore, suppose there are exactly two contiguity subdiagrams, and they are adjacent to sides of distinct colors, say $b_1$ and $r_2$ (other cases only differ by notation). Then Lemma \ref{corner} applied to $p=b_1$ and $q=r_2$ implies that
$$
\d_i((b_1)_-, r_2)\le \d_\Delta ((b_1)_-, r_2)< \ell(b_1)= 3Ni,
$$
which contradicts ($\ast$). Thus, there must exist two contiguity subdiagrams of $\Pi$ to sides of the same color (and possibly other contiguity subdiagrams). For $j=1,2$, we denote by $s_j$ and $t_j$ the side sections of $\partial\Gamma_j$ and $\partial\Sigma_j$ as shown on Fig. \ref{fig4-gon} and consider two subcases.

\smallskip

\noindent{\it Case A.1.\;} Suppose first that $\Gamma_1$ and $\Gamma_2$ are present. By Lemma \ref{cont} (a), we have
$$
\d_i((s_1)_-,(s_2)_+)\le \ell(s_1) +\ell(\partial \Pi)/2 + \ell (s_2) < (2\zeta +1/2)\ell (\partial\Pi) < Ni \;\; (LPP),
$$
which contradicts ($\ast$). Thus, this case is impossible.

\smallskip

\noindent{\it Case A.2.\;} Assume that $\Sigma_1$ and $\Sigma_2$ are present.
As above, we obtain $\d_i((t_1)_-, (t_2)_+) < Ni$. Since the path $r_1$ is geodesic, we have $\ell (r_1) \le \ell(b_1) +\d_i((t_1)_-, (t_2)_+) +\ell (b_2) <7 Ni$. Therefore, $\d_i(o, r_2)\le \ell(r_1)/2 +\ell(b_j) < 6.5Ni$.

\smallskip

\begin{figure}
  \begin{center}
  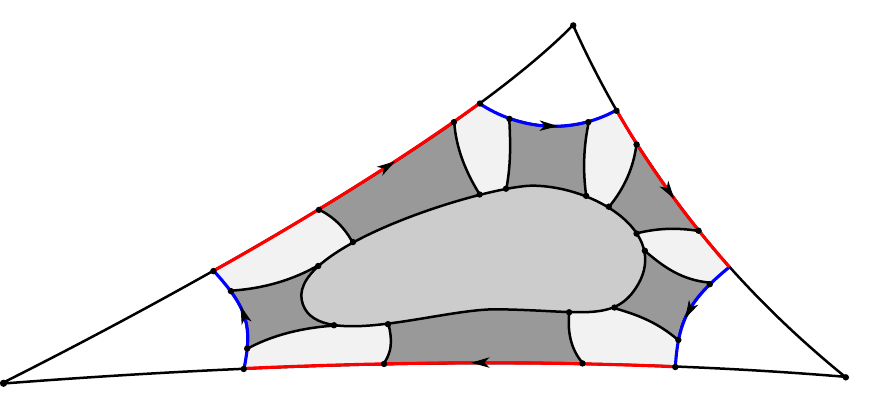
  \end{center}
  \vspace*{-3mm}
  \caption{Case B.}\label{fig6-gon}
 \end{figure}

\noindent{\it Case B.\;} Now suppose that $[z,o]$ (respectively, $[o,x]$) does not intersect the closed $3Ni$-neighborhood of $[x,y]$ (respectively, $[y,z]$).

Let $r_1$ be the shortest subpath of $[z,x]$ containing $o$ such that
$$
\d_i((r_1)_-, [x,y]\cup [y,z])=\d_i((r_1)_+, [x,y]\cup [y,z])=Ni.
$$
Let $u,v\in [x,y]\cup [y,z]$ be the nearest vertices to $(r_1)_+$ and $(r_1)_-$, respectively (see Fig.~\ref{fig6-gon}). The assumption of Case B implies that $u\in [x,y]$, $v \in [y,z]$, and $\d_i(u,v)\ge 2Ni>Ni$. Let $r_2$ (respectively, $r_3$) be the shortest segments of $[x,y]$ (respectively, $[y,z]$) starting at $u$ (respectively, ending at $v$) such that $\d_i((r_2)_+, (r_3)_-)= Ni$. We denote by $b_1$, $b_2$, $b_3$ any geodesic paths in $\Cay(G_i,X)$ connecting $(r_1)_+$ to $u$, $(r_2)_+$ to $(r_3)_-$, and $v$ to $(r_1)_-$, respectively. As above, paths $r_1$, $r_2$, $r_3$ and $b_1$, $b_2$, $b_3$ are called \emph{red} and \emph{blue}, respectively. By construction, the hexagon $r_1b_1r_2b_2r_3b_3$ satisfies the following analog of ($\ast$).

\begin{enumerate}
\item[($\ast\ast$)] The distance in $\Cay(G_i, X)$ between distinct red sides and the length of every blue side are equal to $Ni$.
\end{enumerate}

As above, we consider a reduced van Kampen diagram $\Delta $ over (\ref{Gi}) with boundary $r_1b_1r_2b_2r_3b_3$ (see Fig.~\ref{fig6-gon}) and note that $\Delta $ must contain at least one $\mathcal R$-face. By Lemma \ref{gamma}, there exists an $\mathcal R$-face $\Pi $ and contiguity subdiagrams $\Gamma_j$, $\Sigma_j$ of $\Pi$ to $r_j$ and $b_j$ for $j=1,2,3$ (some of these diagrams may be absent) such that
$$\sum\limits_{j=1}^3\Big((\Pi,\Gamma_j, r_j) + (\Pi, \Sigma_j, b_j)\Big) \ge 1-\gamma.$$

Arguing as in Case A, we show that at least two of these contiguity subdiagrams must be present, and if there are exactly two of them,  they cannot be adjacent to adjacent sides of distinct colors. Furthermore, we cannot have exactly two diagrams adjacent to ``opposite" sides of distinct colors. E.g., if we only have $\Sigma _1$ and $\Gamma_3$, then Lemma \ref{corner} and $(\ast\ast)$ imply that $$\d_i ((b_1)_-, r_3)\le \d_\Delta ((b_1)_-, r_3) \le \ell(b_1)=Ni;$$ this inequality contradicts the assumption that $[o,x]$ does not intersect the closed $3Ni$-neighborhood of $[y,z]$, which we made in Case B. Other cases are treated similarly. Summarising, we conclude that if exactly two contiguity diagrams are present, they must be adjacent to sides of distinct color.  Thus, we are left with only two cases.

\smallskip
\noindent{\it Case B.1.\;}  Two contiguity subdiagrams to red sides are present (here and in Case B.2 below, we do not exclude the possibility of having other contiguity subdiagrams as well). Arguing as in Case A.1, we obtain a contradiction.

\smallskip
\noindent{\it Case B.2.\;} Two contiguity subdiagrams to blue sides are present. First, assume that $\Sigma _1$ and $\Sigma_2$ are present.  Let $s_1$ and $s_2$ be any side sections of $\partial \Sigma_1$ and $\partial\Sigma_2$ (see Fig. \ref{fig6-gon}). By Lemma \ref{cont}~(a), we have
$$
\begin{array}{rcl}
\d_i((r_1)_+, [y,z]) & \le &  \ell(b_1) +\ell(s_1) +\ell(\partial \Pi )/2 +\ell (s_2) +\ell(b_2) \\&&\\ &\le & 2Ni +(2\zeta+1/2) \ell (\partial \Pi)\\&&\\ &< & 3Ni \;\;\; (LPP),
\end{array}
$$
which contradicts the assumption of Case B. Similarly, the case when $\Sigma _2$ and $\Sigma_3$ are present is impossible. Thus, the present contiguity subdiagrams to blue sides must be $\Sigma_1$ and $\Sigma_3$. Arguing as in Case A.2, we obtain $\d_i(o, r_2\cup r_3)\le 6.5 Ni$.

Thus, in all cases, we have (\ref{doxyz}). This implies that $\Cay(G_i, X)$ is hyperbolic with the hyperbolicity constant at most $6.5Ni+1/2< 7Ni$.
\end{proof}

\begin{rem}
A result similar to Proposition \ref{33ni} was proved in \cite[Proposition 6.2 (e)]{OOS} for groups given by a slightly different presentation. The groups considered there have unbounded torsion, which causes the hyperbolicity constant to grow superlinearly in $i$. In fact, the proof given in \cite{OOS} can be modified in order to derive Proposition \ref{33ni}. However, we choose to present a slightly longer yet, in a sense, easier proof. The difference is that the argument given in \cite{OOS} makes use of somewhat more advanced technology based on \emph{$B$-} and \emph{$C$-maps} from \cite[Chapter 7]{book}. In contrast, our approach only uses results about more elementary $A$-maps introduced in \cite[Chapter 5]{book}. This allows us to make our exposition self-contained modulo the input from \cite{book} discussed in the previous section. 
\end{rem}

\begin{thm}\label{Thm:Bmn}
For any $r\ge 1$ and any sufficiently large odd $N$, the free Burnside group $B(r,N)$ is well-approximated by hyperbolic groups. In  particular, $f_{B(r,N)}\sim n/m$.
\end{thm}

\begin{proof}
For every $i\in \NN$, let $j(i)=\lceil i/(1-\beta)\rceil$. We define $\mathcal R^\prime _i= \mathcal R_{j(i)}$ and $G^\prime _i=\langle X\mid \mathcal R^\prime _i\rangle$. By Lemma \ref{beta}, every word $w\in F(X)$ of length $\| w\|\le i$ representing $1$ in $B(r,N)$ belongs to $\ll \mathcal R^\prime_i\rr$. Further, by Proposition \ref{33ni}, $\Cay(G_i^\prime, X)$ is hyperbolic, and its hyperbolicity constant is bounded from above by $7N\lceil i/(1-\beta)\rceil \le 14Ni$ (LPP). Thus, the sequence of presentations  $G^\prime _i$ satisfies Definition \ref{Def:wah} and we obtain the first claim of the theorem. The second claim follows from Proposition \ref{Prop:LimHyp}.
\end{proof}


\subsection{Non-quasi-isometric finitely generated groups of finite exponent}\label{FEQI}


In this section, we prove Theorem \ref{Thm:FE} by applying Proposition \ref{Prop:Cext} to the central extension of the free Burnside group $B(r,N)$ defined as follows. Let $X$ and $\mathcal R$ be as in (\ref{B}) and let $K$ denote the normal closure of $\mathcal R$ in $F(X)$. Fix any integer $p\ge 2$ and consider the group
\begin{equation}\label{Cmnp}
C(r,N,p) =F(X)/K^p[K,F(X)]=\langle X\mid \,R^p=1,\, [R,x]=1,\, R\in \mathcal R,\, x\in X\rangle.
\end{equation}
The lemma below summarizes some properties of $C(r,N,p)$ obtained in \cite{book}.

\begin{lem}[Olshanskii]\label{Cext}
For any integers $r,p\ge 2$ and any sufficiently large odd $N\in \NN$, the center of $C(r,N,p)$ is isomorphic to the direct sum of countably many copies of $\ZZ_p$ and we have $C(r,N,p)/Z(C(r,N,p))\cong B(r,N)$. In particular, $C(r,N,p)$ has exponent $Np$.
\end{lem}
\begin{proof}
By \cite[Corollary 31.1]{book}, the subgroup $K/[F(X),K]$ of $F(X)/[F(X),K]$ is free abelian with basis $\overline{\mathcal R} =\{ \overline {R}\mid R\in \mathcal R\}$, where $\overline{R}=R[F(X), K]$. The set $\mathcal R$ is infinite by \cite[Theorem 19.3]{book}. Therefore, the central subgroup $Z=K/K^p[F(X),K]\le C(r,N,p)$ is isomorphic to the direct sum of countably many copies of $\ZZ_p$ generated by the images of elements $\overline{R}\in \overline{\mathcal R}$. We have $C(r,N,p)/Z\cong B(r,N)$. By \cite[Theorem 19.5]{book}, the group $B(r,N)$ is centerless. Hence,  $Z(C(r,N,p))=Z$.
\end{proof}

Proposition \ref{Prop:Cext} and Lemma \ref{Cext} reduce the proof of Theorem \ref{Thm:FE} to the following.

\begin{prop}\label{mainL}
For any integers $r\ge 4$, $p\ge 2$, and any sufficiently large odd $N\in \NN$, we have $f_{B(r,N)}\not\sim f_{C(r,N,p)}$.
\end{prop}

We will prove this result by showing that the Dehn spectrum of $C(r,N,p)$ is not linear. The assumption $r\ge 4$ is not really essential and is added in order to simplify some technical arguments in the proof; a slightly longer argument would work for any $r\ge 2$.

We begin with auxiliary results. Our first goal is to show that the set of relations in Olshanskii's presentation (\ref{B}) grows exponentially. More precisely, we say that a set $\mathcal W$ of words in some alphabet is \emph{exponential} if there exists $c\in \NN$ such that for all $i\in \NN$, we have
$$
|\{ W\in \mathcal W \mid \| W\| \le ci\} | \ge 2^i.
$$

\begin{lem}[Olshanskii, {\cite[Lemma 1.2]{Ols16}}]\label{C*}
For any $\lambda >0$, there exists an exponential set $\mathcal W$ of words in a $2$-letter alphabet such that the following conditions hold.
\begin{enumerate}
\item[(a)] Every word $W\in \mathcal W$ is $7$-aperiodic; that is, $W$ contains no subwords of the form $U^7$, where $U$ is any non-empty word.
\item[(b)] If $V$ is a common subword of some words $U, W \in \mathcal W$ and $\| V\| \ge \lambda \| W\|$, then $U\equiv W$.
\end{enumerate}
\end{lem}

Recall that the sets of periods $\mathcal P_i$ were defined in the process of constructing the presentation (\ref{B}) in Section \ref{Sec:O}. Let
$$
\mathcal P=\bigcup_{i=1}^\infty \mathcal P_i
$$
The proof of the next lemma is inspired by the proof of \cite[Theorem 19.3]{book}.

\begin{lem}\label{exp}
For every $r\ge 2$ and every sufficiently large odd $N\in \NN$, the set $\mathcal P$ is exponential.
\end{lem}

\begin{proof}
Let $\mathcal W$ be an exponential set of words in the alphabet $X$ satisfying conditions (a) and (b) of Lemma \ref{C*} for $\lambda = 1/40$. For every $W\in \mathcal W$, we have $W^N=1$ in $B(r,N)$. By \cite[Theorem 16.2]{book}, there exists a period $A\in \mathcal P$ such that a cyclic shift of the word $W^N$ and a cyclic shift of $A^N$ have a common subword of length at least $\e N \| A\| >60 \| A\|$ (LPP). We fix one such period for each $W\in \mathcal W$ and denote it by $A_W$. Getting rid of cyclic shifts, we obtain that $A_W^{29}$ is a subword of $W^N$.

Suppose that $A_U\equiv A_W$ for some words $U,W\in \mathcal W$. Without loss of generality, we can assume that $\| W\|\le \|U\|$. If $\| A_W\| \le\|W\|/10$, then $A^{29}$ is a subword of $W^4$, which contradicts the assumption that $W$ is $7$-aperiodic. Therefore, $\| A_W\| >\|W\|/10$.
Since $A_W$ is a subword of $W^N$, the words $A_W$ and $W$ have a common subword $B$ of length at least $\|A_W\|/2> \|W\|/20$. Note that $B$ is also a subword of $U^N$ as $A_U\equiv A_W$. Since $\| B\|\le \| W\|\le \| U\|$, the words $B$ and $U$ share a common subword $C$ of length at least $\| B\|/2>\| W\|/40$. Note that $C$ is also a common subword of $U$ and $W$, and hence we have $U\equiv W$ by condition (b) of Lemma \ref{C*}. Thus, the map $W\mapsto A_W$ is injective. The inequality $\| A_W\| \le N\|W\|$ implies that the subset $\{ A_W\mid W\in \mathcal W\}\subseteq\mathcal P$ is exponential.
\end{proof}

Recall that $E(\Delta)$ denotes the set of edges of a van Kampen diagram $\Delta$. We will also need the following result obtained by Grigorchuk and Ivanov in the course of proving Theorem~1.10 in \cite{GI} (it is stated explicitly in the third paragraph of the proof).

\begin{lem}[{\cite{GI}}]\label{GI}
For every reduced van Kampen diagram $\Delta$ over (\ref{B}), we have $|E(\Delta)|\le \ell(\partial\Delta)^{19/12}$.
\end{lem}

In the proof of Proposition \ref{mainL}, we will borrow some ideas from Chapter 10 of \cite{book} and Section 5.3 of \cite{OOS}. In particular, we will make use of the following invariant. Let $A\in \mathcal P$ be a period. For a diagram $\Delta$ over (\ref{B}), we let $$\sigma _A(\Delta)= \sigma_A^+(\Delta) - \sigma_A^-(\Delta),$$ where $\sigma_A^+(\Delta)$ (respectively $\sigma_A^-(\Delta)$) is the number of faces whose boundary label reads $A^N$ (respectively,  $A^{-N}$) in the counterclockwise direction starting from an appropriate vertex.

Every diagram $\Sigma $ over (\ref{Cmnp}) can be converted to a diagram $\Delta $ over (\ref{B}) with the same boundary label by using the following transformations  (see Fig. \ref{fig1}).

\begin{figure}
  \begin{center}
  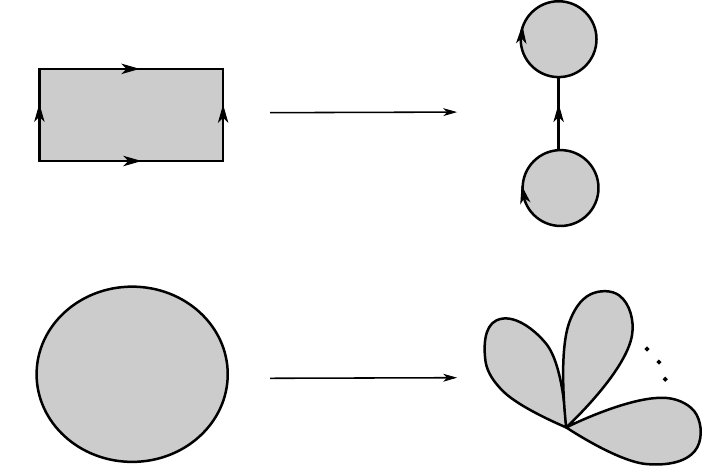
  \end{center}
  \vspace{-3mm}
  \caption{Transformations (T$_1$) and (T$_2$).}\label{fig1}
\end{figure}

\begin{enumerate}
\item[(T$_1$)] Every face with boundary label $[R,x]$ for some $R\in \mathcal R$ and $x\in X$ is replaced with two faces with boundary labels $R$ and $R^{-1}$ (read in the same direction, e.g., clockwise) connected by an edge labelled by $x$.
\item[(T$_2$)] Every face $\Pi$ with boundary label $R^p$ for some $R\in \mathcal R$ is replaced with the wedge of $p$ faces with boundary labels  $R$ (read in the same direction as $\Lab (\partial \Pi)$).
\end{enumerate}

We record an immediate observation.

\begin{lem}\label{conv}
Let $\Sigma $ be any diagram over (\ref{Cmnp}),  $\Delta $ the diagram over (\ref{B}) obtained from $\Sigma$ by transformations (T$_1$) and (T$_2$). For every period $A\in \mathcal P$, we have $\sigma _A(\Delta )\equiv 0 \; ({\rm mod}\, p)$.
\end{lem}

Now, we prove the main technical result of this section.

\begin{lem}\label{qg}
Let $r,p\ge 2$ be any integers and let $N\in \NN$ be a sufficiently large odd number. For any $\ell\in \NN$ and any pairwise distinct periods $A_1, \ldots, A_\ell\in \mathcal P$, we have $|A_1^N\ldots A_\ell^N|_X\ge  (\|A_1^N\ldots A_\ell^N\|/2) ^{12/19}$ in the group $C(r,N,p)$.
\end{lem}

\begin{proof}
Let $S$ be a shortest word representing the same element as $A_1^N\ldots A_\ell^N$ in $C(r,N,p)$. Let $\Sigma $ be a van Kampen diagram over (\ref{Cmnp}) with the boundary label $A_1^N\ldots A_\ell ^NS^{-1}$. We transform $\Sigma $ to a diagram $\Delta$ over (\ref{B}) by using (T$_1$) and (T$_2$). Further, we ``wrap" $\Delta$ around the part of the boundary labelled by $A_1^N\ldots A_\ell ^N$ and fill the obtained hole with a wedge of faces labelled by $A_1^N$, $\ldots$, $A_\ell^N$ (see Fig. \ref{fig3}). Let $\Omega_0$ denote the obtained diagram. Note that $\Omega_0$ may not be reduced. To be able to apply Lemma \ref{GI}, we reduce $\Omega_0$ and denote by $\Omega$ the resulting reduced diagram; this step does not change the boundary of the diagram and thus $\lab(\partial \Omega )\equiv \lab(\partial \Omega_0)\equiv S^{-1}$.

\begin{figure}
  \begin{center}
\hspace*{3mm} 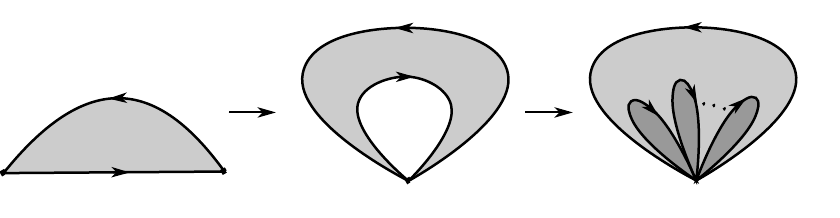
  \end{center}
  \vspace{-3mm}
  \caption{Transforming $\Delta $ to $\Omega _0$.}\label{fig3}
\end{figure}

By Lemma \ref{conv}, we have $\sigma_A (\Delta )=0 \; ({\rm mod}\, p)$ for every period $A$. Since the periods $A_1, \ldots, A_\ell$ are distinct, we have $\sigma_{A_i} (\Omega_0)\equiv 1 \; ({\rm mod}\, p)$ for all $i=1, \ldots, \ell$. The same equality holds for $\Omega$ since the reduction process does not affect the invariant $\sigma _{A_i}$. In particular, the diagram $\Omega $ must contain at least one face labelled by each of the words $A_1^N, \ldots, A_\ell^N$. Therefore, $|E(\Omega)| \ge (\|A_1^N\| +\cdots +\| A_\ell^N\|)/2$. Combining this with Lemma \ref{GI}, we obtain
$$
|A_1^N\ldots A_\ell^N|_X=\| S\| =\ell (\partial\Omega)\ge |E(\Omega)|^{12/19}\ge  (\|A_1^N\ldots A_\ell^N\|/2) ^{12/19}.
$$
\end{proof}

\begin{proof}[Proof of Proposition \ref{mainL}]
Let $X=\{ y_1,y_2,z_1,z_2, ...\}$ be an alphabet of cardinality $r$ and let $B(r,N)$ be the free Burnside group of exponent $N$ with basis $X$. We also denote by $B_Y$ and $B_Z$ the free Burnside groups of exponent $N$ with bases $Y=\{y_1, y_2\}$ and $Z=\{z_1,z_2\}$, respectively.

Since $B(r,N)$ is free in the variety of groups of exponent $N$, there are natural homomorphisms $B(r,N)\to B_Y$ and $B(r,N)\to B_Z$ that send all generators from $X\smallsetminus Y$ (respectively, $X\smallsetminus Z$) to $1$. We denote by $\phi\colon C(r,N,p)\to B_Y$ (respectively, $\psi\colon C(r,N,p)\to B_Z$) the composition of these homomorphisms and the natural homomorphism $C(r,N,p)\to B(r,N)$. Finally, let $$\widehat\phi\colon \Cay(C(r,N,p), X)\to \Cay(B_Y, Y)$$ and $$\widehat\psi\colon \Cay(C(r,N,p), X)\to \Cay(B_Z, Z)$$ be the maps between the Cayley graphs induced by $\phi$ and $\psi$.

Fix some $k\in \NN$ and let $W$ be any word in the alphabet $\{ y_1^{\pm 1}, y_2^{\pm 1}\}$ of length
\begin{equation}\label{W2k}
\|W\|=2^k
\end{equation}
that labels a geodesic path in the Cayley graph $\Cay(B_Y, Y)$. Note that every path labelled by $W$ in $\Cay(C(r,N,p), X)$ is also geodesic. Further, by Lemma \ref{exp}, there is a constant $c\in \NN$ such that the number of periods of length at most $ck$ in Olshanskii's presentation of $B_Z$ is at least $2^k$. Let $A_1, \ldots, A_{2^k}$ be pairwise distinct periods in Olshanskii's presentation of $B_Z$ of length
\begin{equation}\label{Ai}
\| A_i\| \le ck, \;\;\; i=1, \ldots, 2^k,
\end{equation}
There exists $1\le \ell \le 2^k$ such that
\begin{equation} \label{A-side}
2^k\le \| A_1^N \ldots A_\ell^N\| \le 2^k+ckN.
\end{equation}

Let
$$
L\equiv WA_1^N \ldots A_\ell^N W^{-1} (A_1^N \ldots A_\ell^N)^{-1}.
$$
Since every $A_i^N$  belongs to the center of $C(r,N,p)$, $L=1$ in this group. Thus, the word $L$ labels a loop $p=p_1p_2p_3p_4$
in $\Cay(C(r,N,p), X)$, where
$$
\Lab(p_1)\equiv W,\;\;\; \Lab(p_2)\equiv A_1^N \ldots A_\ell^N,\;\;\; \Lab(p_3)\equiv W^{-1},\;\;\; \Lab(p_4)\equiv (A_1^N \ldots A_\ell^N)^{-1}.
$$
Denote the midpoint of $p_1$ by $o$. Note that $o$ is a vertex since $\ell(p_1)=\| W\| =s=2^k$ is even. Let $v$ be a vertex on $p_2p_3p_4$ closest to $o$, let $q$ denote a geodesic path going from $o$ to $v$ in $\Cay(C(r,N,p), X)$, and let $t=\ell(q)$. There are two cases to consider.

\begin{figure}
  \begin{center}
 \hspace*{15mm} 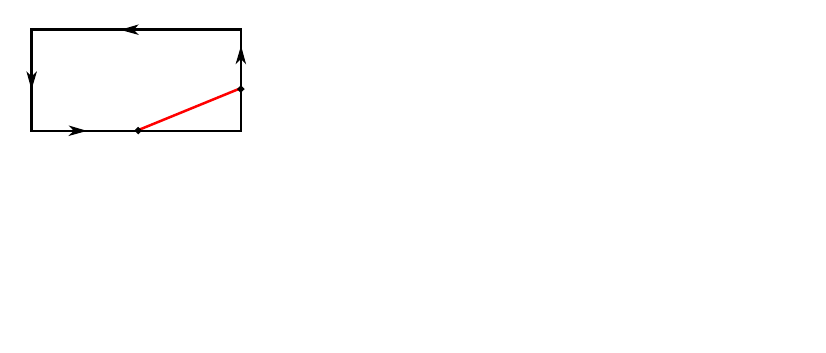
  \end{center}
  \vspace{-3mm}
  \caption{Two cases in the proof of Proposition \ref{mainL}.}\label{fig2}
\end{figure}

\smallskip
\noindent{\emph{Case 1.}}\;  Suppose first that $v$ lies on $p_2$ or $p_4$. To estimate $t$ from below, we project our configuration to $\Cay(B_Y, Y)$ via $\widehat\phi$ (see Fig. \ref{fig2}). We have $\widehat\phi(v)=\widehat\phi(p_1)_+$ (if $v\in p_2$) or $\widehat\phi(v)=\widehat\phi(p_1)_-$ (if $v\in p_4$). Since $\widehat\phi$ maps $p_1$ to a geodesic path of length $\| W\|$ in $\Cay(B_Y, Y)$ by the choice of $W$, we have $$t\ge \| W\| /2=2^{k-1}$$

\smallskip
\noindent{\emph{Case 2.}}\;  Now, assume that $v$ lies on $p_3$. Projecting this configuration to $\Cay(B_Z, Z)$ via the map $\widehat\psi$, we obtain $\widehat\psi (o)=\widehat\psi(p_2)_-$ and $\widehat\psi(v)=\widehat\psi(p_2)_+$. Therefore, $t\ge |A_1^N \ldots A_\ell^N|_Z$ in this case, where $|\cdot |_Z$ denotes the word length in $B_Z$ with respect to $Z$. Applying Lemma~\ref{qg} and using the left inequality in (\ref{A-side}), we obtain $$t\ge(\| A_1^N \ldots A_\ell^N\|/2)^{12/19}\ge 2^{12(k-1)/19}.$$

Note that the relation $L=1$ in $C(r,N,p)$ follows from the relations of the form $[A_i^N, y]=1$, where $i=1, \ldots, \ell$ and $y\in Y$. By (\ref{Ai}), we have
$$
\| [A_i^N, y]\| \le 2 \max\limits_{i=1, \ldots, \ell}\| A^N_i\|+2 \le 2ckN+2<3ckN.
$$
Hence, the loop $p_1p_2p_3p_4$ is $3ckN$-contractible. In both cases considered above, the path $(p_2p_3p_4)^{-1}$ is a $2^{12(k-1)/19}$-detour of the midpoint of $p_1$ (see Definition~\ref{det}). Assume that $C(r,N,p)$ has a linear Dehn spectrum. Applying Proposition \ref{Prop:div} and using (\ref{A-side}) and (\ref{W2k}), we obtain
$$\label{eq on k}
D\left( \frac{2^{12(k-1)/19}}{3ckN} \right)^2 \le \ell(p_2p_3p_4)\le 2\| A_1^N \ldots A_\ell^N\|+\| W\|\le  3\cdot 2^k+2ckN ,
$$
where $D$ is a constant independent of $k$. Clearly, the above inequality on $k$ is nonsense if $k$ is large enough. Thus, the Dehn spectrum of $C(r,N,p)$ is not linear, and the claim of the proposition follows from Theorem \ref{Thm:Bmn} and Theorem \ref{Thm:QI}.
\end{proof}

\begin{rem}\label{Rem:Bow}
We are not aware of any other asymptotic invariant that could be used to distinguish the quasi-isometry classes of $B(r,N)$ and $C(r,N,p)$. For example, the \emph{taut loop length spectrum} introduced by Bowditch \cite{Bow} (which is also called the \emph{relation range} in \cite{BC}) cannot be used here. Indeed, using results of \cite{book}, one can show that $B(r,N)$ and $C(r,N,p)$ are \emph{densely related} in the terminology of \cite{BC}; for $B(r,N)$, this fact can be extracted from the proof of \cite[Theorem 19.3]{book}. This means that the relation ranges of these groups are equivalent and cannot be used to distinguish their quasi-isometry classes. For definitions and details, we refer the interested reader to the papers \cite{BC,Bow}.
\end{rem}

We are now ready to prove Theorem \ref{Thm:FE}, which we state in a more detailed form here. 

\begin{thm}\label{Thm:main-full}
For any integers $r\ge 4$, $p\ge 2$, and any sufficiently large odd $N\in \NN$, there exist $2^{\aleph_0}$ $r$-generated groups of exponent $Np$ with pairwise incomparable Dehn spectra. In particular, there are $2^{\aleph_0}$ quasi-isometry classes of such groups.
\end{thm}

\begin{proof}
Lemma \ref{Cext} and Proposition \ref{mainL} allow us to apply Proposition~\ref{Prop:Cext} to the group $G=C(r,N,p)$ and the first claim of the theorem follows. The second claim follows from the first one and Theorem \ref{Thm:QI}.
\end{proof}

\begin{rem}
It is clear from the proof that the continuum many non-quasi-isometric groups of finite exponent from Theorem \ref{Thm:main-full} have an infinite central subgroup of exponent $p$.
\end{rem}

\section{Open questions}

We conclude our paper with a few open questions. 

The only examples of groups for which we know the Dehn spectrum up to equivalence are those mentioned in the introduction. It would be nice to compute the Dehn spectrum of some other groups, even finitely presented ones.
Specific examples of interest include finitely generated nilpotent groups (in which case we conjecture that $f_G(k,m,n)\sim \delta_G(n)/\delta_G(m)$, at least for free nilpotent groups), the Baumslag-Solitar groups, $SL_3(\ZZ)$, etc. 

Perhaps the most natural question is the realization problem, which is open even for finitely presented groups.

\begin{prob}
    Describe functions that can be realized, up to the equivalence relation introduced in this paper, as Dehn spectra of (a) finitely generated and (b) finitely presented groups.
\end{prob}

Furthermore, the same questions can be asked for specific classes of groups, such as solvable groups, residually finite groups, etc. Conversely, it is natural to ask how the Dehn spectrum of a group influences its algebraic properties. We state one representative question in this direction.

\begin{prob}
    Suppose that $G$ is a finitely generated group with linear Dehn spectrum. Is $G$ well-approximated (or at least approximated) by hyperbolic groups?
\end{prob}

Passing to infinitely presented groups, we suggest the following.

\begin{prob}
Compute the Dehn spectrum of the Grigorchuk group of intermediate growth.
\end{prob}

It is not very difficult to show that the Dehn spectrum of the Grigorchuk group $G$ satisfies the equation $f_G(k,m,n)\sim g(k) (n/m)^2$ for some function $g\colon \mathbb N\to \mathbb N$. However, we do not know whether the equivalence $f_G(k,m,n)\sim (n/m)^2$ holds.

Another interesting class of infinitely presented groups consists of free groups in certain varieties.

\begin{prob}
Compute the Dehn spectrum of finitely generated free solvable groups.
\end{prob}

The answer is not clear even for free metabelian groups. Further, it does not seem to be known whether the non-abelian, free solvable groups of different ranks or different solvability degrees are quasi-isometric. Computing their Dehn spectra may help resolve this question.

Finally, we formulate a natural problem related to our results on free Burnside groups.

\begin{prob}
Let $r,s\ge 2$ and let $M$, $N$ be sufficiently large natural numbers. When are the groups $B(r,M)$ and $B(s, N)$ quasi-isometric? In particular, is it true that $B(r,M)$ is not quasi-isometric to $B(s,N)$ whenever $M$ is even, and $N$ is odd?
\end{prob}

We do not believe that the latter question can be answered by computing the Dehn spectrum; another invariant is likely necessary.

\vspace{5mm}

\noindent \textbf{Denis Osin: } Department of Mathematics, Vanderbilt University, Nashville, TN, U.S.A.

\noindent E-mail: \emph{denis.v.osin@vanderbilt.edu}

\smallskip

\noindent \textbf{Ekaterina Rybak: } Department of Mathematics, Vanderbilt University, Nashville, TN, U.S.A.

\noindent E-mail: \emph{ekaterina.rybak@vanderbilt.edu}

\end{document}